\documentclass{amsart}

\usepackage{amsthm}
\usepackage[leqno]{amsmath}
\usepackage[all]{xy} \SelectTips{eu}{}
\usepackage{latexsym,amsfonts,amssymb}
\usepackage{hyperref}
\usepackage{mathptmx}
\usepackage{nicefrac}
\usepackage{array}
\newcommand{\numberseries}{\bfseries}   %Fontseries used for numbering

\newlength{\thmtopspace}                %Space above theorem
\newlength{\thmbotspace}                %Space below theorem
\newlength{\thmheadspace}               %Space after theorem label
\newlength{\thmindent}                  %For indenting

\newtheoremstyle{fixed bf head,slanted body}
                {\thmtopspace}{\thmbotspace}{\slshape}
                {\thmindent}{\bfseries}{.}{\thmheadspace}
                {{\numberseries \thmnumber{#2\;}}\thmname{#1}\thmnote{ (#3)}}

\newtheoremstyle{variable bf head,slanted body}
                {\thmtopspace}{\thmbotspace}{\slshape}
                {\thmindent}{\bfseries}{.}{\thmheadspace}
                {{\numberseries \thmnumber{#2\;}}\thmname{#1}\thmnote{ #3}}

\newtheoremstyle{fixed bf head,upright body}
                {\thmtopspace}{\thmbotspace}{\upshape}
                {\thmindent}{\bfseries}{.}{\thmheadspace}
                {{\numberseries \thmnumber{#2\;}}\thmname{#1}\thmnote{ (#3)}}

\newtheoremstyle{numbered paragraph}
                {\thmtopspace}{\thmbotspace}{\upshape}
                {\thmindent}{\upshape}{}{\thmheadspace}
                {{\numberseries \thmnumber{#2.}}}

\theoremstyle{fixed bf head,slanted body}
\newtheorem{res}{}[section]
\newtheorem{thm}[res]{Theorem}          \newtheorem*{thm*}{Theorem}
\newtheorem{prp}[res]{Proposition}      \newtheorem*{prp*}{Proposition}
\newtheorem{cor}[res]{Corollary}        \newtheorem*{cor*}{Corollary}
\newtheorem{lem}[res]{Lemma}            \newtheorem*{lem*}{Lemma}

\theoremstyle{variable bf head,slanted body}
     \newtheorem*{introthm*}{Theorem}
   \newtheorem*{introcor*}{Corollary}

\theoremstyle{fixed bf head,upright body}
            \newtheorem*{stp*}{Setup}
\newtheorem{dfn}[res]{Definition}       \newtheorem*{dfn*}{Definition}
     \newtheorem*{con*}{Construction}
      \newtheorem*{obs*}{Observation}
\newtheorem{rmk}[res]{Remark}           \newtheorem*{rmk*}{Remark}
\newtheorem{exa}[res]{Example}          \newtheorem*{exa*}{Example}
\newtheorem{qst}[res]{Question}         \newtheorem*{qst*}{Question}

\theoremstyle{numbered paragraph}
\newtheorem{ipg}[res]{}

\newlength{\thmlistleft}        %leftmargin
\newlength{\thmlistright}       %rightmargin
\newlength{\thmlistpartopsep}   %partopsep
\newlength{\thmlisttopsep}      %topsep
\newlength{\thmlistparsep}      %parsep
\newlength{\thmlistitemsep}     %itemsep

\newcounter{eqc} 
\newenvironment{eqc}{\begin{list}{\upshape (\textit{\roman{eqc}})}%
    {\usecounter{eqc}%
      \setlength{\leftmargin}{\thmlistleft}%
      \setlength{\labelwidth}{\thmlistleft}%
      \setlength{\rightmargin}{\thmlistright}%
      \setlength{\partopsep}{\thmlistpartopsep}%
      \setlength{\topsep}{\thmlisttopsep}%
      \setlength{\parsep}{\thmlistparsep}%
      \setlength{\itemsep}{\thmlistitemsep}}}%
  {\end{list}}%

\newcounter{prt}
\newenvironment{prt}{\begin{list}{\upshape (\alph{prt})}%
    {\usecounter{prt}%
      \setlength{\leftmargin}{\thmlistleft}%
      \setlength{\labelwidth}{\thmlistleft}%
      \setlength{\rightmargin}{\thmlistright}%
      \setlength{\partopsep}{\thmlistpartopsep}%
      \setlength{\topsep}{\thmlisttopsep}%
      \setlength{\parsep}{\thmlistparsep}%
      \setlength{\itemsep}{\thmlistitemsep}}}%
  {\end{list}}%

  \newcommand{\proofoftag}[2][:]{(#2)#1}

\newcommand{\pgref}[1]{\ref{#1}}
\newcommand{\thmref}[2][Theorem~]{#1\pgref{thm:#2}}
\newcommand{\corref}[2][Corollary~]{#1\pgref{cor:#2}}
\newcommand{\prpref}[2][Proposition~]{#1\pgref{prp:#2}}
\newcommand{\lemref}[2][Lemma~]{#1\pgref{lem:#2}}

\newcommand{\dfnref}[2][Definition~]{#1\pgref{dfn:#2}}

\newcommand{\rmkref}[2][Remark~]{#1\pgref{rmk:#2}}
\newcommand{\secref}[2][Section~]{#1\ref{sec:#2}}

\renewcommand{\eqref}[1]{(\pgref{eq:#1})}

\makeatletter
\def\@nobreak@#1{\mathchoice%
  {\nobreakdef@\displaystyle\f@size{#1}}%
  {\nobreakdef@\nobreakstyle\tf@size{\firstchoice@false #1}}%
  {\nobreakdef@\nobreakstyle\sf@size{\firstchoice@false #1}}%
  {\nobreakdef@\nobreakstyle\ssf@size{\firstchoice@false #1}}%
  \check@mathfonts}%
\def\nobreakdef@#1#2#3{\hbox{{%
                    \everymath{#1}%
                    \let\f@size#2\selectfont%
                    #3}}}%
\makeatother

\DeclareFontFamily{T1}{cmex}{}
\DeclareFontShape{T1}{cmex}{m}{n}{<-> s * [0.89] cmex10}{}
\DeclareSymbolFont{cmlargesymbols}{T1}{cmex}{m}{n}
\DeclareMathSymbol{\mycoprod}{\mathop}{cmlargesymbols}{"60} \let\coprod\mycoprod
\DeclareMathSymbol{\myprod}{\mathop}{cmlargesymbols}{"51} \let\prod\myprod

\DeclareSymbolFont{usualmathcal}{OMS}{cmsy}{m}{n}
\DeclareSymbolFontAlphabet{\mathcal}{usualmathcal}

\DeclareSymbolFont{letters}{OML}{txmi}{m}{it}

\DeclareMathSymbol{\alpha}{\mathord}{letters}{"0B}
\DeclareMathSymbol{\beta}{\mathord}{letters}{"0C}
\DeclareMathSymbol{\gamma}{\mathord}{letters}{"0D}
\DeclareMathSymbol{\delta}{\mathord}{letters}{"0E}
\DeclareMathSymbol{\epsilon}{\mathord}{letters}{"0F}
\DeclareMathSymbol{\zeta}{\mathord}{letters}{"10}
\DeclareMathSymbol{\eta}{\mathord}{letters}{"11}
\DeclareMathSymbol{\theta}{\mathord}{letters}{"12}
\DeclareMathSymbol{\iota}{\mathord}{letters}{"13}
\DeclareMathSymbol{\kappa}{\mathord}{letters}{"14}
\DeclareMathSymbol{\lambda}{\mathord}{letters}{"15}
\DeclareMathSymbol{\mu}{\mathord}{letters}{"16}
\DeclareMathSymbol{\nu}{\mathord}{letters}{"17}
\DeclareMathSymbol{\xi}{\mathord}{letters}{"18}
\DeclareMathSymbol{\pi}{\mathord}{letters}{"19}
\DeclareMathSymbol{\rho}{\mathord}{letters}{"1A}
\DeclareMathSymbol{\sigma}{\mathord}{letters}{"1B}
\DeclareMathSymbol{\tau}{\mathord}{letters}{"1C}
\DeclareMathSymbol{\upsilon}{\mathord}{letters}{"1D}
\DeclareMathSymbol{\phi}{\mathord}{letters}{"1E}
\DeclareMathSymbol{\chi}{\mathord}{letters}{"1F}
\DeclareMathSymbol{\psi}{\mathord}{letters}{"20}
\DeclareMathSymbol{\omega}{\mathord}{letters}{"21}
\DeclareMathSymbol{\varepsilon}{\mathord}{letters}{"22}
\DeclareMathSymbol{\vartheta}{\mathord}{letters}{"23}
\DeclareMathSymbol{\varpi}{\mathord}{letters}{"24}
\DeclareMathSymbol{\varrho}{\mathord}{letters}{"25}
\DeclareMathSymbol{\varsigma}{\mathord}{letters}{"26}
\DeclareMathSymbol{\varphi}{\mathord}{letters}{"27}

\DeclareMathSymbol{\Gamma}{\mathord}{letters}{"00}
\DeclareMathSymbol{\Delta}{\mathord}{letters}{"01}
\DeclareMathSymbol{\Theta}{\mathord}{letters}{"02}
\DeclareMathSymbol{\Lambda}{\mathord}{letters}{"03}
\DeclareMathSymbol{\Xi}{\mathord}{letters}{"04}
\DeclareMathSymbol{\Pi}{\mathord}{letters}{"05}
\DeclareMathSymbol{\Sigma}{\mathord}{letters}{"06}
\DeclareMathSymbol{\Upsilon}{\mathord}{letters}{"07}
\DeclareMathSymbol{\Phi}{\mathord}{letters}{"08}
\DeclareMathSymbol{\Psi}{\mathord}{letters}{"09}
\DeclareMathSymbol{\Omega}{\mathord}{letters}{"0A}

\DeclareMathSymbol{\upGamma}{\mathalpha}{operators}{"00}
\DeclareMathSymbol{\upDelta}{\mathalpha}{operators}{"01}
\DeclareMathSymbol{\upTheta}{\mathalpha}{operators}{"02}
\DeclareMathSymbol{\upLambda}{\mathalpha}{operators}{"03}
\DeclareMathSymbol{\upXi}{\mathalpha}{operators}{"04}
\DeclareMathSymbol{\upPi}{\mathalpha}{operators}{"05}
\DeclareMathSymbol{\upSigma}{\mathalpha}{operators}{"06}
\DeclareMathSymbol{\upUpsilon}{\mathalpha}{operators}{"07}
\DeclareMathSymbol{\upPhi}{\mathalpha}{operators}{"08}
\DeclareMathSymbol{\upPsi}{\mathalpha}{operators}{"09}
\DeclareMathSymbol{\upOmega}{\mathalpha}{operators}{"0A}

\newcommand{\cA}{\mathcal{A}}
\newcommand{\cB}{\mathcal{B}}
\newcommand{\cC}{\mathcal{C}}
\newcommand{\cD}{\mathcal{D}}

\newcommand{\cM}{\mathcal{M}}

\newcommand{\e}{\mathsf{e}}
\newcommand{\s}{\mathsf{s}}
\newcommand{\f}{\mathsf{f}}
\newcommand{\g}{\mathsf{g}}
\renewcommand{\c}{\mathsf{c}}
\renewcommand{\k}{\mathsf{k}}

\newcommand{\sou}[1]{s(#1)}
\newcommand{\tar}[1]{t(#1)}
\newcommand{\arrs}[1]{Q_1^{#1\to*}}
\newcommand{\arrt}[1]{Q_1^{*\to#1}}

\newcommand{\Rep}[2]{\operatorname{Rep}(#1,#2)}
\newcommand{\Hom}[3]{\operatorname{Hom}_{#1}(#2,#3)}
\newcommand{\Ext}[4]{\operatorname{Ext}_{#1}^{#2}(#3,#4)}
\newcommand{\Ker}[1]{\operatorname{Ker}#1}
\newcommand{\Coker}[1]{\operatorname{Coker}#1}

\newcommand{\Proj}[1]{\mathrm{Prj}\mspace{2.5mu}#1}
\newcommand{\GProj}[1]{\mathrm{GPrj}\mspace{2.5mu}#1}
\newcommand{\Inj}[1]{\mathrm{Inj}\mspace{2.5mu}#1}
\newcommand{\GInj}[1]{\mathrm{GInj}\mspace{2.5mu}#1}

\begin{document}

\title{Cotorsion pairs in categories of quiver representations}

\author{Henrik Holm \ }
\address{Department of Mathematical Sciences, Universitetsparken 5, University of Copenhagen, 2100 Copenhagen {\O}, Denmark} 
\email{holm@math.ku.dk}
\urladdr{http://www.math.ku.dk/\~{}holm/}

\author{ \ Peter J{\o}rgensen}
\address{School of Mathematics and Statistics, Newcastle University, Newcastle upon Tyne NE1 7RU, United Kingdom} 
\email{peter.jorgensen@ncl.ac.uk}
\urladdr{http://www.staff.ncl.ac.uk/peter.jorgensen}

%\date{\today}

\keywords{Abelian categories; adjoint functors; cotorsion pairs; quivers; quiver representations.}

\subjclass[2010]{18E10, 18G05, 18G15}

%18 Category theory; homological algebra  
%18E10 Exact categories, abelian categories 
%18G05 Projectives and injectives  
%18G15 Ext and Tor, generalizations, Künneth formula  

\begin{abstract}
  We study the category $\Rep{Q}{\cM}$ of representations of a quiver $Q$ with values in an abelian category $\cM$. Under certain assumptions, we show that every cotorsion pair $(\cA,\cB)$ in $\cM$ induces two (explicitly described) cotorsion pairs $(\upPhi(\cA),\Rep{Q}{\cB})$ and $(\Rep{Q}{\cA},\upPsi(\cB))$ in $\Rep{Q}{\cM}$. This is akin to a result by Gillespie, which asserts that a cotorsion pair $(\cA,\cB)$ in $\cM$ induces cotorsion pairs \smash{$(\widetilde{\cA}, \mathrm{dg}\mspace{1mu} \widetilde{\cB})$} and \smash{$(\mathrm{dg}\mspace{1mu}\widetilde{\cA}, \widetilde{\cB})$} in the category $\mathrm{Ch}(\cM)$ of chain complexes in $\cM$. Special cases of our results recover descriptions of the projective and injective objects in $\Rep{Q}{\cM}$ proved by Enochs, Estrada, and Garc{\'{\i}}a~Rozas.   
\end{abstract}

\maketitle

%$\alpha \beta \gamma \delta \epsilon \varepsilon \zeta \eta \theta \vartheta \iota \kappa \lambda \mu \nu o \pi \varpi \rho \varrho \sigma \varsigma \tau \upsilon \phi \varphi \chi \psi \omega$

\vspace*{-2.5ex}

\section{Introduction}
\label{sec:Introduction}

The traditional study of quiver representations is often restricted to  representations with values in the category of modules over a ring (or even in the category of finite dimensional vector spaces over a field). In this paper, we study the category $\Rep{Q}{\cM}$ of $\cM$-valued re\-presentations
of a quiver $Q$ where $\cM$ is an abelian category, and we are interested in how homological properties (here we focus on cotorsion pairs) in $\cM$ carry~over~to~$\Rep{Q}{\cM}$. We extend results from the literature about module-valued quiver representations to general $\cM$-valued representations, but we also prove results about the category $\Rep{Q}{\cM}$ which are new even in the case where $\cM$ is a module category. Our main results, Theorems A and B below, are akin to \cite[Cor.~3.8]{Gillespie} by Gillespie, where it is shown that every cotorsion pair $(\cA,\cB)$ in an abelian category $\cM$ with enough projectives and injectives induces two cotorsion pairs \smash{$(\widetilde{\cA}, \mathrm{dg}\mspace{1mu} \widetilde{\cB})$} and \smash{$(\mathrm{dg}\mspace{1mu}\widetilde{\cA}, \widetilde{\cB})$} in the category $\mathrm{Ch}(\cM)$ of chain complexes in $\cM$; see also \cite{Gillespie07}.

Besides the obvious gain of generality, there is another advantage to working with general $\cM$-valued representations: While it is not true that the opposite of a module category is a module category, it is true that the opposite of an abelian category is abelian. This fact, together with observations like $\Rep{Q^\mathrm{op}}{\cM^\mathrm{op}} = \Rep{Q}{\cM}^\mathrm{op}$ where $Q^\mathrm{op}$ is the opposite quiver of $Q$, makes it easy to dualize results about quiver representations and, in a sense, cuts the work in half. For example, one way to prove Theorem~B below is by applying Theorem~A directly to the opposite quiver $Q^\mathrm{op}$ and the opposite category $\cM^\mathrm{op}$.

We now explain the mathematical content of this paper in more detail. Our work is~mo\-ti\-vated by a series of results about module-valued quiver representations. To explain them, we first need to introduce some notation. For every $i \in Q_0$ (where $Q_0$ denotes the set of~vertices in $Q$) and every $\cM$-valued representation $X$ of $Q$ there are two canonical morphisms,
\begin{displaymath}
  \bigoplus_{a \colon\mspace{-3mu} j\mspace{1mu}\to\mspace{1mu} i} X(j) \stackrel{\varphi^X_i}{\longrightarrow} X(i)
  \qquad \text{and} \qquad
  X(i) \stackrel{\psi^X_i}{\longrightarrow} \prod_{a \colon\mspace{-2mu} i\mspace{1mu}\to j} X(j)\;,    
\end{displaymath}
where the coproduct, respectively, product, is taken over all arrows in $Q$ whose target,~re\-spec\-tively, source, is the vertex $i$. In the following results from the literature, a  ``representation'' means a representation with values in the category of (left) modules over any~ring. 
\begin{prt}

\item[$\bullet$] Enochs and Estrada characterize in \cite[Thm.~3.1]{EEESE05} (2005) the projective representations of a left rooted quiver\footnote{\,The left rooted quivers, which are defined \ref{left-rooted}, constitute quite a large class of quivers.} $Q$. They are exactly the representations $X$ for which the module $X(i)$ is projective and $\varphi^X_i$ is a split monomorphism for every $i \in Q_0$.

\item[$\bullet$] Enochs, Oyonarte, and Torrecillas characterize in \cite[Thm.~3.7]{EOT04} (2004) the flat representations of a left rooted quiver $Q$. They are exactly the representations $X$ for which the module $X(i)$ is flat and $\varphi^X_i$ is a pure monomorphism for every $i \in Q_0$.

\item[$\bullet$] Eshraghi, Hafezi, and Salarian characterize in \cite[Thm.~3.5.1(b)]{EshraghiHafeziSalarian} (2013) the Gorenstein projective representations of a left rooted quiver $Q$. They are exactly the representations $X$ for which the module $X(i)$ is Gorenstein projective and $\varphi^X_i$ is a mono\-morphism with Gorenstein projective cokernel for every $i \in Q_0$.

\end{prt}
As the reader may notice, all these results follow the same pattern. Indeed, if we for a class $\cA$ of objects in $\cM$ define a class $\upPhi(\cA)$ of objects in $\Rep{Q}{\cM}$ by
  \begin{align*}
    \upPhi(\cA) = \left\{\!
    X \in \Rep{Q}{\cM} \left|\!\!\!
    \begin{array}{ll}
      \text{$\varphi^X_i$ is a monomorphism and} \\
      \text{$X(i), \Coker{\varphi^X_i} \in \cA$ for all $i \in Q_0$} 
    \end{array}
    \right.
    \!\!\!\!\!\right\},
  \end{align*}
then the results in \cite{EEESE05,EOT04,EshraghiHafeziSalarian} mentioned above say that if $Q$ is left rooted and $\cA$ is the class of projective, flat, or Gorenstein projective objects in a module category $\cM$, then $\upPhi(\cA)$ is exactly the class of projective, flat, or Gorenstein projective objects in $\Rep{Q}{\cM}$. This indicates that---at least if $Q$ is left rooted---it could be the case  that $\upPhi(\cA)$ will inherit any ``good'' properties which the class $\cA$ might have. Here we study the relationship between $\cA$ and  $\upPhi(\cA)$ from a more abstract point of view. Our focus is on cotorsion pairs, and we prove that if $\cA$ is the left half of a cotorsion pair in $\cM$, and if $Q$ is left rooted, then $\upPhi(\cA)$ will be the left half of a cotorsion pair in $\Rep{Q}{\cM}$. More precisely:

\begin{introthm*}[A]
Let $Q$ be a left rooted quiver and let $\cM$ be an abelian category that satisfies \textnormal{AB4} and \textnormal{AB4*} and which has enough projectives and injectives. If $(\cA,\cB)$ is a cotorsion pair in $\cM$, then there is a cotorsion pair $(\upPhi(\cA),\Rep{Q}{\cB})$ in $\Rep{Q}{\cM}$ where $\upPhi(\cA)$ is defined as above and
  \begin{align*}
    \Rep{Q}{\cB} &= \big\{\,Y \in \Rep{Q}{\cM} \,\big|\, \text{$Y(i) \in \cB$ for all $i \in Q_0$} \big\}\,.    
  \end{align*}
  Moreover, if $(\cA,\cB)$ is hereditary or generated by a set, then so is $(\upPhi(\cA),\Rep{Q}{\cB})$.
\end{introthm*}

\noindent
For the trivial cotorsion pair $(\cA,\cB)=(\Proj{\cM},\cM)$ one has $\Rep{Q}{\cB}=\Rep{Q}{\cM}$, and we get from Theorem~A that the class of projective objects in $\Rep{Q}{\cM}$ is precisely
  \begin{align*}
    \Proj{(\Rep{Q}{\cM})} =
    \upPhi(\Proj{\cM}) = \left\{\!
    X \in \Rep{Q}{\cM} \left|\!\!\!  
    \begin{array}{ll}
      \text{$\varphi^X_i$ is a split monomorphism and} \\
      \text{$X(i) \in \Proj{\cM}$ for all $i \in Q_0$} 
    \end{array}
    \right.
    \!\!\!\!\!\right\}\!.
  \end{align*}
This recovers the result by Enochs and Estrada \cite{EEESE05} mentioned above when $\cM$ is a module  category. We also establish the following dual version of Theorem~A.

\begin{introthm*}[B]
Let $Q$ be a right rooted quiver and let $\cM$ be an abelian category that satisfies \textnormal{AB4} and \textnormal{AB4*} and which has enough projectives and injectives. If $(\cA,\cB)$ is a cotorsion pair in $\cM$, then there is a cotorsion pair $(\Rep{Q}{\cA},\upPsi(\cB))$ in $\Rep{Q}{\cM}$ where
  \begin{align*}
    \Rep{Q}{\cA} &= \big\{X \in \Rep{Q}{\cM} \,\big|\, \text{$X(i) \in \cA$ for all $i \in Q_0$} \big\} \quad \text{and}
    \\   
    \upPsi(\cB) &= 
    \left\{\!
    Y \in \Rep{Q}{\cM} \left|\!\!\!  
    \begin{array}{ll}
      \text{$\psi^Y_i$ is an epimorphism and} \\
      \text{$Y(i),\Ker{\psi^Y_i} \in \cB$ for all $i \in Q_0$} 
    \end{array}
    \right.
    \!\!\!\!\!\right\}.    
  \end{align*}
Moreover, if $(\cA,\cB)$ is hereditary or cogenerated by a set, then so is  $(\Rep{Q}{\cA},\upPsi(\cB))$.  
\end{introthm*}

\noindent
Applied to the other trivial cotorsion pair $(\cA,\cB)=(\cM,\Inj{\cM})$, Theorem~B yields:
  \begin{align*}
    \Inj{(\Rep{Q}{\cM})} =
    \upPsi(\Inj{\cM}) = \left\{\!
    Y \in \Rep{Q}{\cM} \left|\!\!\!  
    \begin{array}{ll}
      \text{$\psi^Y_i$ is a split epimorphism and} \\
      \text{$Y(i) \in \Inj{\cM}$ for all $i \in Q_0$} 
    \end{array}
    \right.
    \!\!\!\!\!\right\}\!.
  \end{align*}
When $\cM$ is a module category, this recovers a result by Enochs, Estrada, and Garc{\'{\i}}a~Rozas \cite[Prop.~2.1, Def.~2.2, and Thm. 4.2]{EEGR09} (2009). We also mention that if $\cB$ is the class of Gorenstein injective modules over a ring, then a result of Eshraghi, Hafezi, and Salarian~\cite[Thm.~3.5.1(a)]{EshraghiHafeziSalarian} (2013) shows that $\upPsi(\cB)$ is precisely the class of Gorenstein injective representations of $Q$, provided that $Q$ is right rooted.

The paper is organized as follows. \secref[Sects.~]{Quivers} and \secref[]{cotorsion} contain preliminaries on quivers~and cotorsion pairs. In \secref[Sect.~]{adj-e} we show the existence of a left adjoint and a right adjoint of the evaluation functor $\e_i$, and in \secref[Sect.~]{adj-s} we do the same for the stalk functor $\s_i$. In \secref[Sect.~]{Ext} we establish some isomorphisms between various Ext groups, which will allow us to describe relevant perpendicular classes in the category of quiver representations. In the final \secref[Sect.~]{main} we prove our main results, including Theorems A and B above.

\section{Quivers}
\label{sec:Quivers}

Throughout this paper, $Q$ is a quiver (that is, a directed graph) with vertex set $Q_0$ and arrow set $Q_1$. Unless anything else is mentioned, there will be no restrictions on the quiver. Thus, $Q$ may have infinitely many vertices, it may have loops and/or oriented cycles, and there may be infinitely many or no arrows from one vertex to another. 

\begin{ipg}
\label{finiteness-conditions}
For $i,j \in Q_0$ (not necessarily different), we write $Q(i,j)$ for the set of paths in $Q$ from $i$ to $j$. The trivial path at vertex $i$ is denoted by $e_i$. For an arrow  arrow \mbox{$a \colon\! i \to j$} in $Q$ we write $\sou{a}$ for its source and $\tar{a}$ for its target, that is, $\sou{a}=i$ and $\tar{a}=j$. For a given vertex $i \in Q_0$ we denote by $\arrs{i}$, respectively, $\arrt{i}$, the set of arrows in $Q$ whose source, respectively, target, is the vertex $i$, that is, 
\begin{displaymath}
  \arrs{i} = \{ a \in Q_1 \,|\, \sou{a}=i\}
  \qquad \text{and} \qquad
  \arrt{i} = \{ a \in Q_1 \,|\, \tar{a}=i\}\;.
\end{displaymath}    
%A quiver $Q$ is called \emph{vertex-finite} if it has only finitely many vertices, i.e.~if the set $Q_0$ is finite. It is called \emph{locally path-finite} if there are only finitely many paths from any given vertex to another, i.e.~if the set $Q(i,j)$ is finite for every pair of vertices $i,j$. It is called \emph{source-finite} if every vertex is the source of at most finitely many arrows, i.e.~if the set $\arrs{i}$ is finite for every vertex $i$. It is called \emph{target-finite} if every vertex is the target of at most finitely many arrows, i.e.~if the set $\arrt{i}$ is finite for every vertex $i$. (Easy examples show that these four notions of finiteness for a quiver $Q$ are independent.) Finally, a quiver is called \emph{finite} if it is both vertex-, locally path-, source-, and target-finite.
\end{ipg}

\begin{ipg}
  \label{e-and-s}
Let $\cM$ be any category. We write $\Rep{Q}{\cM}$ for the category of $\cM$-valued representations of the quiver $Q$. An object $X$ in $\Rep{Q}{\cM}$ assigns to every vertex $i \in Q_0$ an object $X(i)$ in $\cM$ and to every arrow \mbox{$a \colon\! i \to j$} in $Q$ a morphism $X(a) \colon X(i) \to X(j)$ in $\cM$. A morphism $\lambda \colon X \to Y$ in $\Rep{Q}{\cM}$ is a family $\{ \lambda(i) \colon X(i) \to Y(i)\}_{i \in Q_0}$ of morphisms in $\cM$ for which the diagram
\begin{displaymath}
  \xymatrix{
     X(i) \ar[d]_-{X(a)} \ar[r]^-{\lambda(i)} & Y(i) \ar[d]^-{Y(a)}
     \\
     X(j) \ar[r]^{\lambda(j)} & Y(j)
   }
\end{displaymath}   
is commutative for every arrow \mbox{$a \colon\! i \to j$} in $Q$. Note that if $X$ is an object in $\Rep{Q}{\cM}$ and $p \in Q(i,j)$ is a path in $Q$, then by composition $X$ yields a morphism $X(p) \colon X(i) \to X(j)$ in $\cM$. For the trivial path $e_i$, the morphism $X(e_i)$ is the identity on $X(i)$.

For every $i \in Q_0$ there is an \emph{evaluation functor},
\begin{displaymath}
  \xymatrix{
    \Rep{Q}{\cM} \ar[r]^-{\e_i} & \cM
  },
\end{displaymath}
which maps an $\cM$-valued representation $X$ of $Q$ to the object $\e_i(X) = X(i) \in \cM$ at vertex~$i$. 

If $\cM$ has a zero object $0$, then there is also, for every $i \in Q_0$, a \emph{stalk functor},
\begin{displaymath}
  \xymatrix{
    \cM \ar[r]^-{\s_i} & \Rep{Q}{\cM}
  },
\end{displaymath}
which to an object $M \in \cM$ assigns the \emph{stalk representation} $\s_i(M)$ given by $\s_i(M)(j) = 0$ for $j \neq i$ and $\s_i(M)(i) = M$. For every arrow $a \in Q_1$, the morphism $\s_i(M)(a)$ is  zero.
\end{ipg}

\begin{ipg}
  \label{op}
    For a quiver $Q$ we denote by $Q^\mathrm{op}$ its opposite quiver, and for a category $\cC$ we denote by $\cC^\mathrm{op}$ its opposite category. It is straightforward to verify that
\begin{displaymath}
  \Rep{Q^\mathrm{op}}{\cM^\mathrm{op}} = \Rep{Q}{\cM}^\mathrm{op}.
\end{displaymath}  
\end{ipg}  

\begin{ipg}
  \label{inherited-structures}
  If $\cM$ has a certain type of limits (e.g.~products, pullbacks etc.), then $\Rep{Q}{\cM}$ has the same type of limits, and they are computed vertex-wise in $\cM$. A similar remark holds for colimits, cf.~\ref{op}. 
  
  If $\cM$ is abelian, then so is $\Rep{Q}{\cM}$. Kernels, cokernels, and images in $\Rep{Q}{\cM}$ are computed vertex-wise in $\cM$; thus a sequence $X \to Y \to Z$ in $\Rep{Q}{\cM}$ is exact if and only if the sequence $X(i) \to Y(i) \to Z(i)$ is exact in $\cM$ for every vertex $i \in Q_0$. It follows that every evaluation functor $\e_i$ and every stalk functor $\s_i$ is exact.
\end{ipg}

The remaining part of this section is concerned with rooted quivers; this material will not be relevant before \secref{main}.

Left rooted quivers are defined in Enochs, Oyonarte, and Torrecillas \cite[Sect.~3]{EOT04} (where the terminologi ``rooted'' is used instead of ``left rooted'') and the dual notion of right rooted quivers appears in Enochs, Estrada, and Garc{\'{\i}}a Rozas \cite[Sect.~4]{EEGR09}. 

\begin{ipg}
  \label{left-rooted}
 Let $Q$ be any quiver. As in \cite[Sect.~3]{EOT04} we consider the transfinite sequence $\{V_\alpha\}_{\alpha \; \mathrm{ ordinal}}$ of subsets of the vertex set $Q_0$ defined as follows:
\begin{prt}
\item[$\bullet$] For the first ordinal $\alpha=0$ set $V_0 = \emptyset$.

\item[$\bullet$] For a successor ordinal $\alpha=\beta+1$ set\footnote{As $V_0=\emptyset$ it follows that $V_1 = \{ i \in Q_0 \,|\, \textnormal{$i$ is \emph{not} the target of any arrow $a$ in $Q$}\}$. The vertices in $V_1$ are often called \emph{sources} (this includes \emph{isolated vertices}, i.e.~vertices which are neither a source nor a target of any arrow).}
\begin{displaymath}
  \hspace*{6ex} V_{\alpha} \!= V_{\beta+1} \!= \{ i \in Q_0 \,|\, \textnormal{$i$ is \emph{not} the target of any arrow $a$ in $Q$ with $\sou{a} \notin \textstyle\bigcup_{\gamma\leqslant \beta}V_\gamma$}\}.
\end{displaymath}

\item[$\bullet$] For a limit ordinal $\alpha$ set $V_{\alpha} = \bigcup_{\beta<\alpha}V_\beta$.

\end{prt}
Following \cite[Def.~3.5]{EOT04}, a quiver $Q$ is called \emph{left rooted} if there exists some ordinal $\lambda$ such that $V_\lambda=Q_0$. It is proved in \cite[Prop.~3.6]{EOT04} that $Q$ is left rooted if and only if there exists no infinite sequence 
$\cdots \to \bullet \to \bullet \to \bullet$ of (not necessarily different) composable arrows~in~$Q$. Hence, the left rooted quivers constitute quite a large class of quivers, for example, every \emph{path-finite} quiver---that is, a quiver which has only finitely many paths---is left rooted.
\end{ipg}

\enlargethispage{6.5ex}

\begin{exa}
  Let $Q$ be the (left rooted) quiver:
  
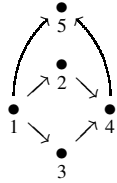
\begin{displaymath}
    \xymatrix@R=-0.1pc@C=0.6pc{
      {} & \underset{5}{\bullet} & {}
      \\
      {} & {} & {}
      \\      
      {} & \underset{2}{\bullet} \ar[dr] & {}
      \\
      \underset{1}{\bullet} \ar[dr] \ar[ur] \ar@/^0.5pc/[uuur]
      & {} & \underset{4}{\bullet} \ar@/_0.5pc/[uuul]
      \\
      {} & \underset{3}{\bullet} \ar[ur] & {}      
    }
  \end{displaymath}
  For this quiver, the transfinite sequence $\{V_\alpha\}$ from \ref{left-rooted} looks like this:
  \begin{displaymath}
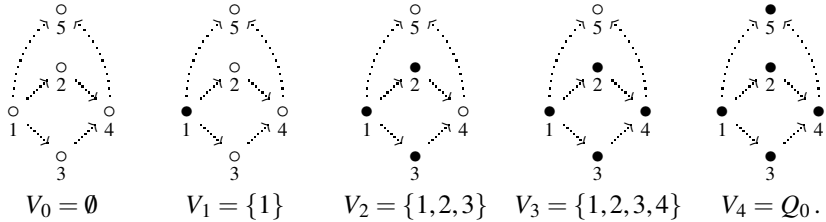
    \begin{array}{ccccc}
    {
    \xymatrix@R=-0.1pc@C=0.6pc{
      {} & \underset{5}{\circ} & {}
      \\
      {} & {} & {}
      \\      
      {} & \underset{2}{\circ} \ar@{..>}[dr] & {}
      \\
      \underset{1}{\circ} \ar@{..>}[dr] \ar@{..>}[ur] \ar@{..>}@/^0.5pc/[uuur]
      & {} & \underset{4}{\circ} \ar@{..>}@/_0.5pc/[uuul]
      \\
      {} & \underset{3}{\circ} \ar@{..>}[ur] & {}      
    }
    }
    & \hspace*{1.8ex}
    {
    \xymatrix@R=-0.1pc@C=0.6pc{
      {} & \underset{5}{\circ} & {}
      \\
      {} & {} & {}
      \\      
      {} & \underset{2}{\circ} \ar@{..>}[dr] & {}
      \\
      \underset{1}{\bullet} \ar@{..>}[dr] \ar@{..>}[ur] \ar@{..>}@/^0.5pc/[uuur]
      & {} & \underset{4}{\circ} \ar@{..>}@/_0.5pc/[uuul]
      \\
      {} & \underset{3}{\circ} \ar@{..>}[ur] & {}      
    }
    }    
    & \hspace*{1.6ex}
    {
    \xymatrix@R=-0.1pc@C=0.6pc{
      {} & \underset{5}{\circ} & {}
      \\
      {} & {} & {}
      \\      
      {} & \underset{2}{\bullet} \ar@{..>}[dr] & {}
      \\
      \underset{1}{\bullet} \ar@{..>}[dr] \ar@{..>}[ur] \ar@{..>}@/^0.5pc/[uuur]
      & {} & \underset{4}{\circ} \ar@{..>}@/_0.5pc/[uuul]
      \\
      {} & \underset{3}{\bullet} \ar@{..>}[ur] & {}      
    }
    }        
    &
    {
    \xymatrix@R=-0.1pc@C=0.6pc{
      {} & \underset{5}{\circ} & {}
      \\
      {} & {} & {}
      \\      
      {} & \underset{2}{\bullet} \ar@{..>}[dr] & {}
      \\
      \underset{1}{\bullet} \ar@{..>}[dr] \ar@{..>}[ur] \ar@{..>}@/^0.5pc/[uuur]
      & {} & \underset{4}{\bullet} \ar@{..>}@/_0.5pc/[uuul]
      \\
      {} & \underset{3}{\bullet} \ar@{..>}[ur] & {}      
    }
    }            
    &
    {
    \xymatrix@R=-0.1pc@C=0.6pc{
      {} & \underset{5}{\bullet} & {}
      \\
      {} & {} & {}
      \\      
      {} & \underset{2}{\bullet} \ar@{..>}[dr] & {}
      \\
      \underset{1}{\bullet} \ar@{..>}[dr] \ar@{..>}[ur] \ar@{..>}@/^0.5pc/[uuur]
      & {} & \underset{4}{\bullet} \ar@{..>}@/_0.5pc/[uuul]
      \\
      {} & \underset{3}{\bullet} \ar@{..>}[ur] & {}      
    }
    }                
    \\
    V_0=\emptyset
    & \hspace*{1.8ex}
    V_1 = \{1\}
    & \hspace*{1.6ex}
    V_2 = \{1,2,3\}
    &
    V_3 = \{1,2,3,4\}    
    &
    V_4 = Q_0 \,.       
    \end{array}
  \end{displaymath}  
\end{exa}

The following properties about the transfinite sequence $\{V_\alpha\}$ from \ref{left-rooted}---which we will need later---are not mentioned in \cite{EOT04}, however, these properties are probably known to the authors of \cite{EOT04}. A consequence of the lemma below is that one can simplify the definition of $V_{\beta+1}$ in \ref{left-rooted} to be $V_{\beta+1} \!= \{ i \in Q_0 \,|\, \textnormal{$i$ is \emph{not} the target of any arrow $a$ in $Q$ with $\sou{a} \notin V_\beta$}\}$.

\begin{lem}
  \label{lem:Ascending}
  The transfinite sequence $\{V_\alpha\}$ defined in \ref{left-rooted} is ascending, that is, for every pair of ordinals $\alpha$, $\beta$ with $\alpha<\beta$ one has $V_\alpha \subseteq V_\beta$. In particular, one has \smash{$\bigcup_{\alpha\leqslant \beta}V_\alpha=V_\beta$} for every ordinal $\beta$.
\end{lem}

\begin{proof}
  It suffices, for every ordinal $\gamma$, to prove the statement:
  \begin{displaymath}
  \tag{$P_\gamma$}
  \textnormal{For every pair of ordinals $\alpha$, $\beta \leqslant \gamma$ for which $\alpha<\beta$ one has $V_\alpha \subseteq V_\beta$.}
  \end{displaymath}
  We will do this by transfinite induction on $\gamma$. The induction start is easy: For $\gamma=0$ the statement is empty since the situation $\alpha<\beta\leqslant \gamma=0$ is impossible. And for $\gamma=1$ the only possibility for $\alpha<\beta\leqslant \gamma=1$ is $\alpha=0$ and $\beta=1$, and evidently $V_0 \subseteq V_1$ as $V_0 = \emptyset$.
  
  Now assume that $\gamma$ is a limit ordinal and that ($P_\delta$) holds for all $\delta<\gamma$. To prove that ($P_\gamma$) is true, let ordinals $\alpha<\beta \leqslant \gamma$ be given. If $\beta<\gamma$ then, as ($P_\beta$) holds, we get that $V_\alpha \subseteq V_\beta$. If $\beta=\gamma$, then one has $V_\beta = V_\gamma = \bigcup_{\delta<\gamma}V_\delta$ (since $\gamma$ is a limit ordinal), so clearly $V_\alpha \subseteq V_\beta$.
  
  It remains to consider the situation where $\gamma=\delta+1$ is a successor ordinal. We assume that ($P_\delta$) holds and must show that ($P_{\delta+1}$) holds as well. Let ordinals $\alpha<\beta \leqslant \delta+1$ be given. If one has $\beta<\delta+1$, then $\beta \leqslant \delta$ and it follows from ($P_\delta$) that $V_\alpha \subseteq V_\beta$. Now assume that $\beta=\delta+1$. As $\alpha \leqslant\delta$ and since ($P_\delta$) holds, we have $V_\alpha \subseteq V_\delta$. Thus, to prove the desired conclusion $V_\alpha \subseteq V_\beta = V_{\delta+1}$, it suffices to argue that $V_\delta \subseteq V_{\delta+1}$. There are two cases:
  
  (1) $\delta$ is a limit ordinal. To prove $V_\delta \subseteq V_{\delta+1}$, assume that $j \in V_\delta$. As $\delta$ is a limit ordinal, we have $V_{\delta} = \bigcup_{\sigma<\delta}V_\sigma$ and hence $j \in V_\sigma$ for some $\sigma<\delta$. Since $\sigma<\sigma+1<\delta$ and since ($P_\delta$) holds, we have $V_\sigma \subseteq V_{\sigma+1}$ and therefore also $j \in V_{\sigma+1}$. By definition, this means that there exists \emph{no} arrow \mbox{$i \to\! j$} in $Q$ with 
  $i \notin \bigcup_{\tau\leqslant \sigma}V_\tau$. As $\sigma \leqslant \delta$ (in fact, $\sigma<\delta$) one has \smash{$\bigcup_{\tau\leqslant \sigma}V_\tau \subseteq \bigcup_{\tau\leqslant \delta}V_\tau$}, and it follows that there exists \emph{no} arrow \mbox{$i \to\! j$} in $Q$ with $i \notin \bigcup_{\tau\leqslant \delta}V_\tau$. By definition, this means that $j \in V_{\delta+1}$, as desired.
  
  (2) $\delta=\varepsilon+1$ is a successor ordinal.  To prove $V_\delta \subseteq V_{\delta+1}$, assume that $j \in V_\delta = V_{\varepsilon+1}$. By definition, this means that there exists \emph{no} arrow \mbox{$i \to\! j$} in $Q$ with 
  $i \notin \bigcup_{\tau\leqslant \varepsilon}V_\tau$. As $\varepsilon \leqslant \delta$ (in fact, $\varepsilon<\delta$) one has \smash{$\bigcup_{\tau\leqslant \varepsilon}V_\tau \subseteq \bigcup_{\tau\leqslant \delta}V_\tau$}, and it follows that there exists \emph{no} arrow \mbox{$i \to\! j$} in $Q$ with $i \notin \bigcup_{\tau\leqslant \delta}V_\tau$. By definition, this means that $j \in V_{\delta+1}$ as desired.  
\end{proof}

\begin{cor}
  \label{cor:no-arrow}
  Let $i, j \in Q_0$ and let $\{V_\alpha\}$ be the transfinite sequence from \ref{left-rooted}. If $i \notin V_\alpha$ and $j \in V_{\alpha+1}$ (in particular, if $j \in V_\alpha$ by \lemref{Ascending}), then there exists no arrow \mbox{$i \to\! j$} in~$Q$.
\end{cor}

\begin{proof}
  Since $j \in V_{\alpha+1}$ there exists  by definition \emph{no} arrow \mbox{$k \to\! j$} in $Q$ with \smash{$k \notin \bigcup_{\beta\leqslant \alpha}V_\beta$}. By \lemref{Ascending} we have \smash{$\bigcup_{\beta\leqslant \alpha}V_\beta=V_\alpha$}, so there exists \emph{no} arrow \mbox{$k \to\! j$} in $Q$ with $k \notin V_\alpha$.
\end{proof}

\begin{ipg}
  \label{right-rooted}
 Let $Q$ be a quiver. As in \cite[Sect.~4]{EEGR09} we consider the transfinite sequence $\{W_\alpha\}_{\alpha \; \mathrm{ ordinal}}$ of subsets of the vertex set $Q_0$ defined as follows:
\begin{prt}
\item[$\bullet$] For the first ordinal $\alpha=0$ set $W_0 = \emptyset$.

\item[$\bullet$] For a successor ordinal $\alpha=\beta+1$ set\footnote{Actually, in \cite[Sect.~4]{EEGR09} they set $W_{\beta+1} \!= \{ i \in Q_0 \,|\, \textnormal{$i$ is \emph{not} the source of any arrow $a$ in $Q$ with $\tar{a} \notin W_\beta$}\}$, but this is the same as the definition of $W_{\beta+1}$ we have given; cf.~the text preceding \lemref{Ascending}.}
\begin{displaymath}
  \hspace*{6ex} W_{\alpha} \!= W_{\beta+1} \!= \{ i \in Q_0 \,|\, \textnormal{$i$ is \emph{not} the source of any arrow $a$ in $Q$ with $\tar{a} \notin \textstyle\bigcup_{\gamma\leqslant \beta}W_\gamma$}\}.
\end{displaymath}

\item[$\bullet$] For a limit ordinal $\alpha$ set $W_{\alpha} = \bigcup_{\beta<\alpha}W_\beta$.

\end{prt}
A quiver $Q$ is called \emph{right rooted} if there exists some ordinal $\lambda$ such that $W_\lambda=Q_0$, equivalently, if there exists no infinite sequence 
$\bullet \to \bullet \to \bullet \to \cdots$ of (not necessarily different) composable arrows in $Q$. 
\end{ipg}

Note that the sequence $\{V_\alpha\}$ in \ref{left-rooted} for the quiver $Q^\mathrm{op}$ coincides with the sequence $\{W_\alpha\}$ in \ref{right-rooted} for the quiver $Q$. Therefore a quiver $Q$ is left rooted, respectively, right rooted, if and only if the opposite quiver $Q^\mathrm{op}$ is right rooted, respectively, left rooted.

\section{Adjoints of the evaluation functor $\e_i$}
\label{sec:adj-e}

As stated in \secref{Quivers}, we work with an arbitrary quiver $Q$. Furthermore, in this section, $\cM$ denotes any category. We will show that if $\cM$ has small coproducts, respectively, small products, then the evaluation functor $\e_i \colon \Rep{Q}{\cM} \to \cM$ from \ref{e-and-s} has a left adjoint $\f_i$, respectively, a right adjoint $\g_i$. If \mbox{$\cM = \mathrm{Mod}\mspace{3mu}R$} is the category of (left) modules over a ring $R$, then the left adjoint of $\e_i$ was constructed in Enochs, Oyonarte, and Tor\-re\-cillas \cite{EOT04} and the right adjoint of $\e_i$ was considered in Enochs and Herzog \cite{EnochsHerzog}. Here we give a shorter and cleaner argument which works for any category $\cM$, and also explains the duality between the functors $\f_i$ and $\g_i$; cf.~\ref{g}.

\begin{ipg}  
  \label{f}
  Assume that $\cM$ has small coproducts and fix any vertex $i \in Q_0$. For any $M \in \cM$ we construct a representation $\f_i(M) \in \Rep{Q}{\cM}$ as follows. For $j \in Q_0$ set 
\begin{displaymath}
  \f_i(M)(j) =  \textstyle{\coprod_{p \in Q(i,j)}M_p} 
\end{displaymath}
where each $M_p$ is a copy of $M$. Notice that if there are no paths in $Q$ from $i$ to $j$, then this coproduct is empty and hence $\f_i(M)(j)$ is the initial object in $\cM$. Let \mbox{$a \colon\! j \to k$} be an arrow in $Q$. Note that each path $p \in Q(i,j)$ yields a path $ap \in Q(i,k)$, and we define $\f_i(M)(a)$ to be the unique morphism in $\cM$ that makes the following diagrams commutative:
\begin{equation}
  \label{eq:f}
  \begin{gathered}
  \xymatrix@C=1pc{
    M_p \ar[d]_-{\varepsilon_p} \ar@{=}[r] & 
    M \ar@{=}[r] & 
    M_{ap} \ar[d]^-{\varepsilon_{ap}}
    \\
    \f_i(M)(j) \ar[rr]^-{\f_i(M)(a)} & & \f_i(M)(k)
  }
  \end{gathered}  
  \qquad (p \in Q(i,j))\,.  
\end{equation}
Here the vertical morphisms $\varepsilon_*$ are the canonical injections. It is evident that the assignment $M \mapsto \f_i(M)$ yields a functor $\f_i \colon \cM \to \Rep{Q}{\cM}$.
\end{ipg}  

\begin{rmk}
  \label{rmk:fewer}
  For the construction of the functors $\f_i$ to work, it is not necessary to require that $\cM$ has \textsl{all} small coproducts; it suffices to assume that the coproduct exists in $\cM$ for every set of objects $\{M_u\}_{u \in U}$ with cardinality $|U| = |Q(i,j)|$ for some $i,j \in Q_0$. 
  
  A quiver $Q$ is called \emph{locally path-finite} if there are only finitely many paths in $Q$ from any given vertex to another, i.e.~if the set $Q(i,j)$ is finite for all $i,j \in Q_0$. For such a quiver, the functors $\f_i \colon \cM \to \Rep{Q}{\cM}$ exist for every category $\cM$ with \textsl{finite} coproducts.
\end{rmk}

\begin{exa}
  \label{exa:f1}
  Let $Q$ be the quiver with one vertex (labelled ``$1$'') and one loop:
  
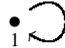
\begin{displaymath}
    \xymatrix{
      %*{{}_1\bullet\mspace{15mu}} 
      \underset{1}{\bullet} \ar@(ur,dr)[]
    }
  \end{displaymath}
  Using ``element notation'', the functor $\f_1$ maps $M \in \cM$ to the representation
  \begin{displaymath}
    \xymatrix{
      *{M \ \text{\raisebox{0.3ex}{$\scriptstyle\coprod$}} \ M \ \text{\raisebox{0.3ex}{$\scriptstyle\coprod$}} \ M \ \text{\raisebox{0.3ex}{$\scriptstyle\coprod$}} \ \cdots \mspace{130mu}} \ar@(ur,dr)[]^-{\lambda}
    }
    \mspace{-70mu}
    \quad \text{where} \quad
    \lambda(m_0,m_1,m_2,\ldots) = (0,m_0,m_1,\ldots)\;.
  \end{displaymath}  
  Note that the functor $\f_1$ exists if $\cM$ has countable coproducts, cf.~\rmkref{fewer}.
\end{exa}

\begin{exa}
  \label{exa:f2}
  Let $Q$ be the quiver 
  \begin{displaymath}
    A_\infty = \ \
    \xymatrix@C=1.35pc{
      \cdots \ar[r] & \underset{i+2}{\bullet} \ar[r] & \underset{i+1}{\bullet} \ar[r] & \underset{i}{\bullet} \ar[r] & \underset{i-1}{\bullet} \ar[r] & \cdots \ar[r] & \underset{2}{\bullet} \ar[r] & \underset{1}{\bullet}
    }.
  \end{displaymath}
  The functor $\f_i$ maps $M \in \cM$ to the representation
  \begin{displaymath}
    \phantom{A_\infty = \ \ }
    \xymatrix@C=1.2pc{
      \cdots \ar[r] & \underset{i+2}{0} \ar[r] & \underset{i+1}{0} \ar[r] & \underset{i}{M} \ar[r]^-{=} & \underset{i-1}{M} \ar[r]^-{=} & \cdots \ar[r]^-{=} & \underset{2}{M} \ar[r]^-{=} & \underset{1}{M}
    }.
  \end{displaymath}
  where $0$ is the initial object in $\cM$. Note that for this particular quiver, the only requirement for the existence of $\f_i$ is that $\cM$ has an initial object (= empty coproduct), cf.~\rmkref{fewer}.
\end{exa}

\begin{lem}
  \label{lem:f-path}
  For $i \in Q_0$ and $M \in \cM$ consider the representation $\f_i(M) \in \Rep{Q}{\cM}$ constructed in \ref{f}. For every path $p \in Q(i,j)$ one has $\f_i(M)(p)\circ\varepsilon_{e_i}=\varepsilon_p$.
\end{lem}

\begin{proof}
  The assertion is obviously true for the trivial path $p=e_i$ as $\f_i(M)(e_i)$ is the identity morphism. Every non-trivial path $p$ from $i$ to $j$ is a finite sequence of arrows in $Q$,
  \begin{displaymath}
    \xymatrix@C=1.3pc{
      i = j_1 \ar[r]^-{a_1} & j_2 \ar[r]^-{a_2} & \cdots \ar[r]^-{a_n} & j_{n+1} = j
    } \qquad (n \geqslant 1)
  \end{displaymath}
and the desired identity follows from successive applications of \eqref{f}.
\end{proof}
      
\begin{ipg}  
  \label{g}
  Assume that $\cM$ has small products and fix any vertex $i \in Q_0$. By a construction dual to that in \ref{f} one gets a functor $\g_i \colon \cM \to \Rep{Q}{\cM}$, that is, for $j \in Q_0$ we have
\begin{displaymath}
  \g_i(M)(j) =  \textstyle{\prod_{\,q \in Q(j,i)}M_q}
\end{displaymath}
where each $M_q$ is a copy of $M$. If there are no paths in $Q$ from $j$ to $i$, then this product is empty and hence $\g_i(M)(j)$ is the terminal object in $\cM$. For an arrow \mbox{$a \colon\! j \to k$} in $Q$ the morphism $\g_i(M)(a)$ is the unique one that makes the following diagrams commutative:
\begin{equation*}
%  \label{eq:g}
  \begin{gathered}
  \xymatrix@C=1pc{
    \g_i(M)(j) \ar[d]_-{\pi_{qa}} \ar[rr]^-{\g_i(M)(a)} & & 
    \g_i(M)(k) \ar[d]^		-{\pi_q}
    \\
    M_{qa} \ar@{=}[r] & M \ar@{=}[r] & M_q
  }
  \end{gathered}  
  \qquad (q \in Q(k,i))\,.  
\end{equation*}
Here the vertical morphisms $\pi_*$ are the canonical projections.

Let us make the duality between the functors $\f_i$ ang $\g_i$ even more clear: A precise notation for the functor $\f_i \colon \cM \to \Rep{Q}{\cM}$ in \ref{f} is \smash{$\f_i^{Q,\cM}$}, and it exists for every quiver $Q$ and every category $\cM$ with small coproducts. If $\cM$ has small products, then  $\cM^\mathrm{op}$ has small coproducts, and thus it makes sense to consider the functor \smash{$\f_i^{Q^\mathrm{op},\cM^\mathrm{op}} \colon  \cM^\mathrm{op} \to \Rep{Q^\mathrm{op}}{\cM^\mathrm{op}}$}.
By taking the opposite of this functor, see \cite[Chap.~II\S2]{Mac}, one gets in view of \ref{op} a functor
\begin{displaymath}
  (\f_i^{Q^\mathrm{op},\cM^\mathrm{op}})^\mathrm{op} \colon \cM \longrightarrow \Rep{Q}{\cM}\;,
\end{displaymath} 
and it is straightforward to verify that this functor is nothing but $\g_i$ ($=\g_i^{Q,\cM}$).
\end{ipg}    
  
\begin{thm}
  \label{thm:adjoints-of-e}
  Let $\cM$ be any category, let $i$ be any vertex in a quiver $Q$, and consider the evaluation functor $\e_i \colon \Rep{Q}{\cM} \to \cM$ from \ref{e-and-s}. The following assertions hold.
\begin{prt}
\item If $\cM$ has small coproducts, then the 
functor $\f_i$ from \ref{f} is a left adjoint of~$\e_i$.

\item If $\cM$ has small products, then the functor $\g_i$ from \ref{g} is a right adjoint of $\e_i$.
\end{prt}
\end{thm}

\begin{proof}
  \proofoftag{a} For $M \in \cM$ and $X \in \Rep{Q}{\cM}$ we construct a pair of natural maps
  \begin{displaymath}
    \xymatrix{
      \Hom{\Rep{Q}{\cM}}{\f_i(M)}{X} \ar@<0.5ex>[r]^-{u} & 
      \Hom{\cM}{M}{\e_i(X)} \ar@<0.5ex>[l]^-{v}
    }
  \end{displaymath}  
  as follows. The map $u$ sends a morphism $\lambda \colon \f_i(M) \to X$ of representations to the morphism $u(\lambda):=\lambda(i)\circ\varepsilon_{e_i}$ in $\cM$, that is, the composition of the morphisms
  \begin{equation}
    \label{eq:u}
    \xymatrix{
      M = M_{e_i} \ar[r]^-{\varepsilon_{e_i}} & \coprod_{p \in Q(i,i)} M_p = \f_i(M)(i) \ar[r]^-{\lambda(i)} & X(i) = \e_i(X)\;,
    }
  \end{equation}  
  where $e_i$ is the trivial path at vertex $i$. To define the map $v$, let $\alpha \colon M \to \e_i(X) = X(i)$ be a morphism in $\cM$. For every vertex $j \in Q_0$ we define a morphism $\lambda(j) \colon \f_i(M)(j) \to X(j)$ as follows. If there are no paths from $i$ to $j$, then $Q(i,j)$ is empty and hence $\f_i(M)(j)$ is the initial object in $\cM$. In this case, $\lambda(j)$ is the unique morphism from the initial object~to~$X(j)$. Suppose that there exists a path from $i$ to $j$.  Any such path $p \in Q(i,j)$ yields~a~morphism $X(p) \colon X(i) \to X(j)$, and we define $\lambda(j)$ to be the unique morphism that makes the following diagrams commutative:
\begin{equation}
  \label{eq:lambda-1}
  \begin{gathered}
  \xymatrix{
     M \ar[d]_-{\varepsilon_p} \ar[r]^-{\alpha} & X(i) \ar[d]^-{X(p)}
     \\
     \f_i(M)(j) \ar[r]^{\lambda(j)} & X(j)\,.
   }
  \end{gathered}    
  \qquad (p \in Q(i,j))
\end{equation}   
To see that the hereby constructed family $\{\lambda(j)\}_{j \in Q_0}$ yields a morphism of representations $v(\alpha):=\lambda \colon \f_i(M) \to X$, we must argue that for every arrow \mbox{$a \colon\! j \to k$} in $Q$, the diagram
  \begin{equation}
    \label{eq:morphism}
    \begin{gathered}
    \xymatrix@C=2.5pc{
      \f_i(M)(j) \ar[d]_{\lambda(j)} \ar[r]^-{\f_i(M)(a)} & \f_i(M)(k) \ar[d]^{\lambda(k)}
      \\
      X(j) \ar[r]^-{X(a)} & X(k)
    }
    \end{gathered}    
  \end{equation}   
  is commutative. This is clear if there are no paths from $i$ to $j$, as in this case $\f_i(M)(j)$ is the initial object in $\cM$. If there exists some path from $i$ to $j$, then commutativity of \eqref{morphism} amounts, by the universal property of the coproduct, to showing that $X(a)\circ\lambda(j)\circ\varepsilon_p = \lambda(k)\circ\f_i(M)(a)\circ\varepsilon_p$ for every $p \in Q(i,j)$. This follows from the defining properties \eqref{lambda-1} of~$\lambda$ and \eqref{f} of $\f_i(M)$, indeed, one has:
\begin{displaymath}  
   X(a)\circ\lambda(j)\circ\varepsilon_p = X(a)\circ X(p)\circ\alpha = X(ap)\circ\alpha = \lambda(k)\circ\varepsilon_{ap} = \lambda(k)\circ\f_i(M)(a)\circ\varepsilon_p\;.
\end{displaymath}  

It is clear that the hereby constructed maps $u$ and $v$ are natural in $M$ and $X$, and it remains to prove that they are inverses of each other: 

Let $\alpha \colon M \to X(i)$ be a morphism and set $\lambda:=v(\alpha)$. By \eqref{u} the morphism $u(\lambda)=uv(\alpha)$ is $\lambda(i)\circ\varepsilon_{e_i}$, which by \eqref{lambda-1} is $X(e_i)\circ\alpha = \alpha$. Hence the composition $uv$ is the identity. 

Conversely, let $\lambda \colon \f_i(M) \to X$ be a morphism and set $\alpha:=u(\lambda)=\lambda(i)\circ\varepsilon_{e_i}$. To prove that $\tilde{\lambda} := v(\alpha) = vu(\lambda)$ is equal to $\lambda$, it must be argued that $\tilde{\lambda}(j)$ and $\lambda(j)$ is the same morphism $\f_i(M)(j) \to X(j)$ for every $j \in Q_0$. If there are no paths from $i$ to $j$, then $\f_i(M)(j)$ is the initial object in $\cM$, so evidently $\tilde{\lambda}(j)=\lambda(j)$. If there exists a path from $i$ to $j$, then for every such path $p \in Q(i,j)$ we have
\begin{displaymath}
  \tilde{\lambda}(j)\circ\varepsilon_p = X(p)\circ\alpha = 
  X(p)\circ\lambda(i)\circ\varepsilon_{e_i} = \lambda(j)\circ\f_i(M)(p)\circ\varepsilon_{e_i} = \lambda(j)\circ\varepsilon_p\;, 
\end{displaymath}
where the first equality is by the defining property \eqref{lambda-1} of $\tilde{\lambda} = v(\alpha)$, the second equality is by the definition of $\alpha$, the third equality holds as $\lambda$ is a morphism of quiver representations, and the fourth and last equality follows from \lemref{f-path}. By the universal property of the coproduct, we now conclude that $\tilde{\lambda}(j)=\lambda(j)$.

\proofoftag{b} Consider the evaluation functor \smash{$\e_i=\e_i^{Q,\cM} \colon \Rep{Q}{\cM} \to \cM$}. In view of \ref{op}, its opposite functor \smash{$(\e_i^{Q,\cM})^\mathrm{op}$} can be identified with the evaluation functor
\begin{displaymath}
  \e_i^{Q^\mathrm{op},\cM^\mathrm{op}} \colon \Rep{Q^\mathrm{op}}{\cM^\mathrm{op}} 
  \longrightarrow \cM^\mathrm{op}\;.
\end{displaymath}  
By part (a), this functor has a left adjoint, namely \smash{$\f_i^{Q^\mathrm{op},\cM^\mathrm{op}}$}\!, so it follows from \lemref{adjoint} below that the functor \smash{$(\f_i^{Q^\mathrm{op},\cM^\mathrm{op}})^\mathrm{op}$} is a right adjoint of $\e_i=\e_i^{Q,\cM}$. However, \smash{$(\f_i^{Q^\mathrm{op},\cM^\mathrm{op}})^\mathrm{op}$} is equal to $\g_i$ by \ref{g}.
\end{proof}

\begin{lem}
  \label{lem:adjoint}
  Let $F \colon \cC \to \cD$ be a functor. If the opposite functor $F^\mathrm{op} \colon \cC^\mathrm{op} \to \cD^\mathrm{op}$ has a left adjoint $G \colon \cD^\mathrm{op} \to \cC^\mathrm{op}$, then the functor $G^\mathrm{op} \colon \cD \to \cC$ is a right adjoint of $F$.
\end{lem}                                                          

\begin{proof}
  As $G$ is a left adjoint of $F^\mathrm{op}$, there is a bijection $\Hom{\cC^\mathrm{op}}{GY}{X} \cong \Hom{\cD^\mathrm{op}}{Y}{F^\mathrm{op}X}$, which is natural in $X \in \cC$ and $Y \in \cD$.
  By the definitions, this is the same as a bijection $\Hom{\cC}{X}{G^\mathrm{op}Y} \cong \Hom{\cD}{FX}{Y}$,  which expresses that $G^\mathrm{op}$ is a right adjoint of $F$.
\end{proof}

It is convenient to recall some of Grothendieck's axioms for abelian categories.

\begin{ipg}
\label{Grothendiecks-axioms}
An abelian category satisfies AB3 if it has small coproducts, equivalently, if it is co\-com\-plete. It satisfies AB4 it if satisfies AB3 and any coproduct of monomorphisms is a monomorphism. The axioms AB$3^*$ and AB$4^*$ are dual to AB3 and AB4.
\end{ipg}

As noted in \ref{inherited-structures}, the category $\Rep{Q}{\cM}$ inherits various types of categorical properties from $\cM$. The next result, which is a consequence of \thmref{adjoints-of-e}, has the same flavor.

\begin{cor}
  \label{cor:enough-projectives}
  Let $\cM$ be any abelian category and let $Q$ be any quiver.
  \begin{prt}
  \item \hspace*{-0.4ex}Assume that $\cM$ satisfies \textnormal{AB3}. If $\cM$ has enough projectives, then so does $\Rep{Q}{\cM}$.
  
  \item \hspace*{-0.4ex}Assume that $\cM$ satisfies \textnormal{AB3$^*$}. If $\cM$ has enough injectives, then so does $\Rep{Q}{\cM}$.  
  
  \end{prt}
\end{cor}

\begin{proof}
  \proofoftag{a} As explained in \ref{inherited-structures}, each evaluation functor $\e_i$ is exact, and by \thmref{adjoints-of-e} it has a left adjoint $\f_i$. It follows that if $P$ is a projective object in $\cM$, then $\f_i(P)$ is projective in $\Rep{Q}{\cM}$ since the functor $\Hom{\Rep{Q}{\cM}}{\f_i(P)}{-} \cong \Hom{\cM}{P}{\e_i(-)}$ is exact. Now, let $X$ be any object in $\Rep{Q}{\cM}$. Since $\cM$ has enough projectives there exists for each $i \in Q_0$ an epimorphism $\pi_i \colon P_i \twoheadrightarrow X(i)=\e_i(X)$ in $\cM$ with $P_i$ projective. Let $\rho$ be the unique morphism in $\Rep{Q}{\cM}$ that makes the following diagrams commutative:
\begin{displaymath}  
  \begin{gathered}
  \xymatrix{
    \f_i(P_i) \ar[d]_-{\textnormal{$i^\mathrm{th}$ injection}} \ar[r]^-{\f_i(\pi_i)} & \f_i\e_i(X) \ar[d]^-{\epsilon_i^X}
    \\
    \bigoplus_{j \in Q_0}\f_j(P_j) \ar[r]^-{\rho} & X
  }
  \end{gathered}
  \qquad (i \in Q_0)\,,
\end{displaymath}  
where $\epsilon_i$ is the counit of the adjunction $\f_i \dashv \e_i$. As noted above, each $\f_j(P_j)$ is projective in $\Rep{Q}{\cM}$ and hence so is the coproduct $\bigoplus_{j \in Q_0}\f_j(P_j)$. We claim that $\rho$ is an epimorphism. It suffices to show that $\rho(i)=\e_i(\rho)$ is an epimorphism for every $i \in Q_0$, as cokernels in $\Rep{Q}{\cM}$ are computed vertex-wise; see \ref{inherited-structures}. By applying $\e_i$ to the diagram above, we see that $\e_i(\rho)$ will be an epimorphism if $\e_i(\epsilon_i^X)\circ \e_i\f_i(\pi_i)$ is an epimorphism. However, $\e_i\f_i(\pi_i)$ is an epimorphism as $\pi_i$ is an epimorphism and the functor $\e_i\f_i$ is right exact (as already noted, $\e_i$ is exact, and $\f_i$ is right exact since it is a left adjoint). And it is well-known, see e.g.~\cite[Chap.~IV, Thm.~1]{Mac}, that $\e_i(\epsilon_i^X)$ is a split epimorphism with right-inverse \smash{$\eta_i^{\e_i(X)}$}, where
$\eta_i$ is the unit of the adjunction $\f_i \dashv \e_i$.

  \proofoftag{b} Dual to the proof of (a). Alternatively, apply part (a) to the opposite quiver $Q^\mathrm{op}$ and the opposite category $\cM^\mathrm{op}$ and invoke \ref{op}.
\end{proof}

\section{Adjoints of the stalk functor $\s_i$}
\label{sec:adj-s}

As stated in \secref{Quivers}, we work with an arbitrary quiver $Q$. Furthermore, in this section, $\cM$ denotes any abelian category. We will show that if $\cM$ satisfies AB3, respectively, AB3$^*$ (see \ref{Grothendiecks-axioms}), then the stalk functor $\s_i \colon \cM \to \Rep{Q}{\cM}$ from \ref{e-and-s} has a left adjoint $\c_i$, respectively, a right adjoint $\k_i$. For the next construction, recall the notation from \ref{finiteness-conditions}.

\begin{ipg}
  \label{c}
  Assume that $\cM$ satisfies AB3 and fix any vertex $i \in Q_0$. For each $X \in \Rep{Q}{\cM}$ we denote by $\varphi^X_i$ the unique morphism in $\cM$ that makes the following diagrams commutative:
\begin{displaymath}
  \begin{gathered}
  \xymatrix{
    X(\sou{a}) \ar[d]_-{\varepsilon_a} \ar[dr]^-{X(a)} & {} \\
    \bigoplus_{a \in \arrt{i}} X(\sou{a}) \ar[r]_-{\varphi^X_i} & X(i)
  }
  \end{gathered}  
  \qquad (a \in \arrt{i})\,.
\end{displaymath}
Here $\varepsilon_a$ denotes the canonical injection. It is clear that the assignment $X \mapsto \varphi^X_i$ is a functor from $\Rep{Q}{\cM}$ to the category of morphisms in $\cM$, and thus one has a functor
\begin{displaymath}
  \c_i = \c_i^{Q,\cM} \colon \Rep{Q}{\cM} \longrightarrow \cM
  \qquad \text{given by} \qquad X \longmapsto \Coker{\varphi^X_i}\;.
\end{displaymath}
\end{ipg}

\begin{rmk}
  \label{rmk:target-finite}
  For the construction of the functors $\c_i$ to work, it is not necessary to require that $\cM$ has \textsl{all} small coproducts; it suffices to assume that the coproduct exists in $\cM$ for every set of objects $\{M_u\}_{u \in U}$ with cardinality $|U| = |\arrt{i}|$ for some $i \in Q_0$. 
  
  A quiver $Q$ is called \emph{target-finite} if every vertex in $Q$ is the target of at most finitely many arrows, that is, if the set $\arrt{i}$ is finite for every vertex $i$. For such a quiver, the~functors $\c_i \colon \cM \to \Rep{Q}{\cM}$ exist for \textsl{any} abelian category $\cM$.
\end{rmk}

\begin{exa}
    Let $Q$ be the quiver 
  \begin{displaymath}
    \xymatrix@C=1.2pc{
      \underset{1}{\bullet} \ar@<0.5ex>[r] \ar@<-0.5ex>[r] & \underset{2}{\bullet}}.
  \end{displaymath}
  For an $\cM$-valued representation \smash{$X = \xymatrix@C=1.2pc{X(1) \ar@<0.5ex>[r]^-{\alpha} \ar@<-0.5ex>[r]_-{\beta} & X(2)}$}of $Q$ we have
  \begin{displaymath}
    \c_1(X) = \Coker\big(0 \to X(1)\big) = X(1)
    \quad \text{ and } \quad
    \c_2(X) = \Coker\bigg(
    \xymatrix{{\def\arraystretch{0.9}\begin{matrix} X(1) \\ \oplus \\ X(1) \end{matrix}} \ar[r]^-{\left(\begin{smallmatrix} \alpha & \beta \end{smallmatrix}\right)} & X(2)}\!\!\!\bigg)\,.
  \end{displaymath}
  For this quiver, the functors $\c_1$ and $\c_2$ exist for any abelian category $\cM$; cf.~\rmkref{target-finite}.
\end{exa}

\begin{ipg}
  \label{k}
  Assume that $\cM$ satisfies AB3$^*$ and fix any vertex $i \in Q$. For each $X \in \Rep{Q}{\cM}$ we denote by $\psi^X_i$ the unique morphism in $\cM$ that makes the following diagrams commutative:
\begin{displaymath}
  \begin{gathered}
  \xymatrix{
    X(i) \ar[dr]_-{X(a)} \ar[r]^-{\psi^X_i} & 
    \prod_{a \in \arrs{i}} X(\tar{a}) \ar[d]^-{\pi_a} \\
    {} & X(\tar{a})\,.
  }
  \end{gathered}  
  \qquad (a \in \arrs{i})
\end{displaymath}
Here $\pi_a$ denotes the canonical projection. It is clear that we get a functor
\begin{displaymath}
  \k_i = \k_i^{Q,\cM} \colon \Rep{Q}{\cM} \longrightarrow \cM
  \qquad \text{given by} \qquad X \longmapsto \Ker{\psi^X_i}\;.
\end{displaymath}
In analogy with the considerations in \ref{g}, one sees that $\k_i=\k_i^{Q,\cM}$ is equal to $(\c_i^{Q^\mathrm{op},\cM^\mathrm{op}})^\mathrm{op}$.
\end{ipg}

\begin{thm}
  \label{thm:adjoints-of-s}
  Let $\cM$ be any abelian category, let $i$ be any vertex in a quiver $Q$, and consider the stalk functor $\s_i \colon \cM \to \Rep{Q}{\cM}$ from \ref{e-and-s}. The following assertions hold.
\begin{prt}
\item If $\cM$ satisfies \textnormal{AB3}, then the functor $\c_i$ from \ref{c} is a left adjoint of $\s_i$.

\item If $\cM$ satisfies \textnormal{AB3$^*$}, then the functor $\k_i$ from \ref{k} is a right adjoint of $\s_i$.

\end{prt}
\end{thm}

\begin{proof}
  \proofoftag{a} For $X \in \Rep{Q}{\cM}$ and $M \in \cM$ we construct below a pair of natural maps
  \begin{displaymath}
    \xymatrix{
      \Hom{\cM}{\c_i(X)}{M} \ar@<0.5ex>[r]^-{u} & 
      \Hom{\Rep{Q}{\cM}}{X}{\s_i(M)} \ar@<0.5ex>[l]^-{v}
    }.
  \end{displaymath}  
  By definition, see \ref{c}, one has $\c_i(X)=\Coker{\varphi^X_i}$, so there is a right exact sequence,
  \begin{displaymath}
    \xymatrix{
      \bigoplus_{a \in \arrt{i}} X(\sou{a}) \ar[r]^-{\varphi^X_i} & 
      X(i) \ar[r]^-{\rho^X_i} & \c_i(X) \ar[r] & 0
    },
  \end{displaymath}  
  where \smash{$\rho^X_i$} is the canonical morphism. 
  
  The map $u$ sends a morphism $\alpha \colon \c_i(X) \to M$ in $\cM$ to the morphism $\lambda \colon X \to \s_i(M)$ defined as follows: For every $j \in Q_0$ with $j \neq i$ one has $\s_i(M)(j)=0$ and we set $\lambda(j)=0$. One also has $\s_i(M)(i)=M$ and we set $\lambda(i) = \alpha\rho^X_i$. We must argue that $\lambda$ is a morphism of representations of $Q$, that is, we must show that  $\lambda(k)\circ X(a) = \s_i(M)(a)\circ\lambda(j)$ for every arrow \mbox{$a \colon\! j \to k$}. Since $\s_i(M)(a)=0$ (always) and $\lambda(k)=0$ for $k\neq i$, the only thing that needs to be checked is that $\lambda(i)\circ X(a)=0$ for all arrows \mbox{$a \colon\! j \to i$}, that is, for all $a \in \arrt{i}$. However, for every such arrow $a$ we have by definition $\lambda(i)\circ X(a)=\alpha\rho^X_i\varphi^X_i\varepsilon_a=\alpha\mspace{1mu}0\mspace{1mu}\varepsilon_a=0$.
  
  For a morphism $\lambda \colon X \to \s_i(M)$ in $\Rep{Q}{\cM}$ we have 
$\lambda(k)\circ X(a) = 0$ for every arrow \mbox{$a \colon\! j \to k$} in $Q$. In particular, the morphism $\lambda(i) \colon X(i) \to M$ satisfies $\lambda(i)\circ\varphi^X_i\varepsilon_a=\lambda(i)\circ X(a)=0$ for every $a \in \arrt{i}$. By the universal property of the coproduct, it follows that $\lambda(i)\circ\varphi^X_i=0$. Thus by the universal property of the cokernel, $\lambda(i)$ factors uniquely through the morphism $\rho^X_i \colon X(i) \to \c_i(X)=\Coker{\varphi^X_i}$. That is, there exists a unique morphism \smash{$\overline{\lambda(i)} \colon \c_i(X) \to M$} such that \smash{$\overline{\lambda(i)}\circ\rho^X_i = \lambda(i)$}. We define $v(\lambda)$ to be this morphism \smash{$\overline{\lambda(i)}$}.

It is clear that the hereby constructed maps $u$ and $v$ are natural in $X$ and $M$, and that they are inverses of each other.

\proofoftag{b} Dual to the proof of (a). Alternatively, in view of \ref{k} and \lemref{adjoint}, part (b)~fol\-lows directly by applying (a) to the opposite quiver $Q^\mathrm{op}$ and the opposite category $\cM^\mathrm{op}$.\!~\qedhere
\end{proof}

\section{Isomorphisms of groups of extensions}
\label{sec:Ext}

In this section, we extend the adjunctions in \thmref[Theorems~]{adjoints-of-e} and \thmref[]{adjoints-of-s} to the level of Ext. The following lemma is the key to our results.

\begin{lem}
  \label{lem:Ext1-general}
  Let $F \colon \cA \to \cB$ and $G \colon \cB \to \cA$ be functors between abelian categories where $F$ is a left adjoint of $G$. Fix an integer $n \geqslant 0$ and objects $A \in \cA$ and $B \in \cB$. Assume that:
\begin{prt}
\item[\textnormal{(1)}] The functor $F$ maps every exact sequence $0 \to GB \to D_1 \to \cdots \to D_n \to A \to 0$ in~$\cA$ to an exact sequence $0 \to FGB \to FD_1 \to \cdots \to FD_n \to FA \to 0$, and

\item[\textnormal{(2)}] The functor $G$ maps every exact sequence $0 \to B \to E_1 \to \cdots \to E_n \to FA \to 0$ in $\cB$ to an exact sequence $0 \to GB \to GE_1 \to \cdots \to GE_n \to GFA \to 0$.

\end{prt}  
Then there is an isomorphism of abelian groups, $\Ext{\cB}{n}{FA}{B} \cong \Ext{\cA}{n}{A}{GB}$.
\end{lem}

\begin{proof}
By the assumptions, the functors $F$ and $G$ yield well-defined group \mbox{homomorphisms} $F(-) \colon \Ext{\cA}{n}{A}{GB} \to \Ext{\cB}{n}{FA}{FGB}$ and $G(-) \colon \Ext{\cB}{n}{FA}{B} \to \Ext{\cA}{n}{GFA}{GB}$. Let $\eta$ be the unit and let $\varepsilon$ be the counit of the adjunction $F \dashv G$ and consider the group homomorphisms $u$ and $v$ given by the following compositions:
\begin{displaymath}
  \begin{gathered}
  \xymatrix@C=-2pc{
    \Ext{\cB}{n}{FA}{B} \ar[dr]_(0.4){G(-)} \ar[rr]^-{u} & & \Ext{\cA}{n}{A}{GB}
    \\
    & \Ext{\cA}{n}{GFA}{GB} \ar[ur]_(0.6){\Ext{\cA}{n}{\eta_A}{GB}} &
  }
  \end{gathered}  
  \qquad \text{and} \qquad
  \begin{gathered}
  \xymatrix@C=-2pc{
    \Ext{\cB}{n}{FA}{B} & & \Ext{\cA}{n}{A}{GB} \ar[ll]_-{v} \ar[dl]^(0.4){F(-)}
    \\
    & \Ext{\cB}{n}{FA}{FGB} \ar[ul]^(0.6){\Ext{\cB}{n}{FA}{\varepsilon_B}}\,. &
  }
  \end{gathered}    
\end{displaymath}  
It is tedious but straightforward to verify that $u$ and $v$ are inverses of each other (for $n=0$ this is a well-known fact, see \cite[Chap.~IV.1, Thm.~1]{Mac}), and we leave it as an exercise.
\end{proof}

The next result concerns the evaluation functor $\e_i$ and its adjoints $\f_i$ and $\g_i$ (\secref{adj-e}). 

\begin{prp}
  \label{prp:Ext-e}
  Let $\cM$ be any abelian category and let $i$ be any vertex in a quiver $Q$. 
\begin{prt}
\item Assume that $\cM$ satisfies \textnormal{AB4}. For all objects $M \in \cM$ and $X \in \Rep{Q}{\cM}$ and all integers $n \geqslant 0$ there is an isomorphism,
\begin{displaymath}
  \Ext{\Rep{Q}{\cM}}{n}{\f_i(M)}{X} \cong 
  \Ext{\cM}{n}{M}{\e_i(X)}\,.
\end{displaymath}  

\item Assume that $\cM$ satisfies \textnormal{AB4$^*$}. For all objects $M \in \cM$ and $X \in \Rep{Q}{\cM}$ and all integers $n \geqslant 0$ there is an isomorphism,
\begin{displaymath}
  \Ext{\Rep{Q}{\cM}}{n}{X}{\g_i(M)} \cong 
  \Ext{\cM}{n}{\e_i(X)}{M}\,.
\end{displaymath}  
\end{prt}
\end{prp}

\begin{proof}
  \proofoftag{a} As $\cM$ satisfies AB3, the left adjoint $\f_i$ of $\e_i$ exists by \thmref{adjoints-of-e}. The functor $\f_i$ is certainly right exact, as it is a left adjoint, but it is even exact: this follows directly from the construction \ref{f} of $\f_i$ and the assumption AB4 that any coproduct of monomorphisms is a monomorphism. The asserted isomorphism now follows from \lemref{Ext1-general}.
  
  \proofoftag{b} Dual to the proof of (a). Alternatively, apply part (a) directly to the opposite quiver $Q^\mathrm{op}$ and the opposite category $\cM^\mathrm{op}$.
\end{proof}

\begin{rmk}
  For the conclusion in \prpref{Ext-e}(a) to hold, it is not always necessary to require that $\cM$ satisfies AB4. For example, if $Q$ is a locally path finite quiver, then the functor $\f_i$ exists and it is exact for \textsl{any} abelian category $\cM$, cf.~\rmkref{fewer}.
\end{rmk}

The next result concerns the stalk functor $\s_i$ and its adjoints $\c_i$ and $\k_i$ (\secref{adj-s}).

\begin{prp}
  \label{prp:Ext-s}
  Let $\cM$ be any abelian category and let $i$ be any vertex in a quiver $Q$. 
  \begin{prt}
  
  \item Assume that $\cM$ satisfies \textnormal{AB3}. Let $X \in \Rep{Q}{\cM}$ be a representation for which~$\varphi^X_i$ is a monomorphism and let $M \in \cM$ be any object. Then there is an isomorphism,
\begin{displaymath}
  \Ext{\Rep{Q}{\cM}}{1}{X}{\s_i(M)} \cong 
  \Ext{\cM}{1}{\c_i(X)}{M}\,.
\end{displaymath}  

  \item Assume that $\cM$ satisfies \textnormal{AB3$^*$}. Let $X \in \Rep{Q}{\cM}$ be a representation for which~$\psi^X_i$ is an epimorphism and let $M \in \cM$ be any object. Then there is an isomorphism,
\begin{displaymath}
  \Ext{\Rep{Q}{\cM}}{1}{\s_i(M)}{X} \cong 
  \Ext{\cM}{1}{M}{\k_i(X)}\,.
\end{displaymath}  

    \end{prt}

\end{prp}

\begin{proof}
  \proofoftag{a} We will apply \lemref{Ext1-general} with $n=1$ to the adjunction $\c_i \dashv \s_i$ from \thmref{adjoints-of-s}. The functor $\s_i$ is exact so it satisfies the hypothesis in \lemref{Ext1-general}(2). To see that $\c_i$ satisfies \lemref[]{Ext1-general}(1) we must argue that $\c_i$ maps every short exact sequence $0 \to \s_i(M) \to D \to X \to 0$
  in $\Rep{Q}{\cM}$ to a short exact sequence in $\cM$ (this is not true for any $X$, but we shall see that it is true in our case where \smash{$\varphi^X_i$} is assumed to be a monomorphism). Such a short exact sequence induces the following commutative diagram in $\cM$ with exact rows,
\begin{equation}
  \label{eq:SL}
  \begin{gathered}
  \xymatrix@C=1.3pc{
    {}
    & 
    \bigoplus_{a \in \arrt{i}} \s_i(M)(\sou{a}) \ar[d]^-{\varphi^{ \s_i(M)}_i} \ar[r] 
    & 
    \bigoplus_{a \in \arrt{i}} D(\sou{a}) \ar[d]^-{\varphi^D_i} \ar[r]
    & 
    \bigoplus_{a \in \arrt{i}} X(\sou{a}) \ar[d]^-{\varphi^X_i} \ar[r] 
    & 0
    \\
    0 \ar[r] & \s_i(M)(i) \ar[r] & D(i) \ar[r] & X(i) \ar[r] & 0.
  }
  \end{gathered}
\end{equation}    
(We are not guaranteed that a coproduct of monomorphisms in $\cM$ is a monomorphism, as we have not assumed that $\cM$ satisfies AB4. Thus, the left-most morphism in the top row of \eqref{SL} is not necessarily monic.) By assumption, \smash{$\Ker{\varphi^X_i}=0$}, so the exact kernel-cokernel sequence that arises from applying the Snake Lemma to \eqref{SL} shows that the sequence
\begin{displaymath}
  0 \longrightarrow \Coker{\varphi^{\s_i(M)}_i} \longrightarrow \Coker{\varphi^D_i} \longrightarrow \Coker{\varphi^X_i} \longrightarrow 0
\end{displaymath}
is exact. 
%(Note that \smash{$\Coker{\varphi^{\s_i(M)}_i}\!=M$} since \smash{$\varphi^{\s_i(M)}_i$} is the zero morphism and  $\s_i(M)(i)=M$.) 
By definition, this sequence is nothing but $0 \to \c_i\s_i(M) \to \c_i(D) \to \c_i(X) \to 0$, and since $\c_i\s_i(M) \cong M$ this completes the proof.

  \proofoftag{b} Dual to the proof of (a). Alternatively, apply part (a) directly to the opposite quiver $Q^\mathrm{op}$ and the opposite category $\cM^\mathrm{op}$.
\end{proof}

\enlargethispage{5ex}

\begin{ipg}
  \label{cone}
  Fix objects $X \in \Rep{Q}{\cM}$ and $M \in \cM$ and fix a vertex $i \in Q_0$. Given any family $\Xi = \{\xi_a\}_{a \in \arrt{i}}$ of morphisms $\xi_a \colon X(\sou{a}) \to M$ in $\cM$ we construct a representation
  \begin{displaymath}
    C = C(X,M,i,\Xi) \in \Rep{Q}{\cM}
  \end{displaymath}
  as follows. For a vertex $j \in Q_0$ we set
  \begin{displaymath}
    C(j) = X(j) \ \ \text{ for } \ \ j \neq i
    \qquad \text{ and } \qquad
    C(j) = C(i) = 
    {\def\arraystretch{0.9}
    \begin{matrix}
      X(i) \\ \oplus \\ M
    \end{matrix}
    }
    \ \ \text{ for } \ \ j=i\;.
  \end{displaymath}  
  The morphism $C(a) \colon C(j) \to C(k)$  associated to an arrow
  \mbox{$a \colon\! j \to k$} in $Q$ is, depending on four different cases, defined as shown in the following table: \vspace*{1ex}
  \begin{displaymath}
    \begin{array}{|c|c|}
      \hline
      \textnormal{($1^\circ$) If $j \neq i$ and $k \neq i$\,:} & 
      \textnormal{($3^\circ$) If $j \neq i$ and $k=i$\,:}
      \\
      \xymatrix@C=3pc{X(j) \ar[r]^-{X(a)} & X(k)} & 
      \xymatrix@C=3pc{X(j) \ar[r]^-{\left(\begin{smallmatrix} X(a) \\ \xi_a \end{smallmatrix}\right)} & {\def\arraystretch{0.9}\begin{matrix} X(i) \\ \oplus \\ M \end{matrix}}} 
      \\ \hline 
      \textnormal{($2^\circ$) If $j=i$ and $k \neq i$\,:} &
      \textnormal{($4^\circ$) If $j=i$ and $k=i$\,:}
      \\
      \xymatrix@C=3pc{{\def\arraystretch{0.9}\begin{matrix} X(i) \\ \oplus \\ M \end{matrix}} \ar[r]^-{\left(\begin{smallmatrix} X(a) & 0 \end{smallmatrix}\right)} & X(k)} & 
      \xymatrix@C=3pc{{\def\arraystretch{0.9}\begin{matrix} X(i) \\ \oplus \\ M \end{matrix}} \ar[r]^-{\left(\begin{smallmatrix} X(a) & 0 \\ \xi_a & 0 \end{smallmatrix}\right)} & {\def\arraystretch{0.9}\begin{matrix} X(i) \\ \oplus \\ M \end{matrix}}}  
      \\ \hline    
    \end{array}
   \vspace*{1ex}    
  \end{displaymath}
  The hereby constructed representation $C$ fits into a short exact sequence in $\Rep{Q}{\cM}$,
\begin{equation}
  \label{eq:ses}
  \xymatrix@C=1.5pc{
    0 \ar[r] & \s_i(M) \ar[r]^-{\iota} & C \ar[r]^-{\pi} & X \ar[r] & 0
  }, 
\end{equation}  
where $\iota(j)$ and $\pi(j)$ are defined as follows:
\begin{displaymath}
  \begin{array}{ccc}
  \textnormal{For $j\neq i$\,:} & & \textnormal{For $j=i$\,:}
  \\
  {
  \xymatrix@R=2.0pc@C=2.5pc{
    \s_i(M)(j) \ar@{=}[d] \ar[r]^-{\iota(j)} & C(j) \ar@{=}[d] \ar[r]^-{\pi(j)} & X(j) \ar@{=}[d]
    \\
    0 \ar[r]^-{0} & X(j) \ar[r]^-{1_{X(j)}} & X(j)
  }
  } 
  & \quad \quad &
  {
  \xymatrix@R=0.9pc@C=2.5pc{
    \s_i(M)(i) \ar@{=}[d] \ar[r]^-{\iota(i)} & C(i) \ar@{=}[d] \ar[r]^-{\pi(i)} & X(i)\phantom{\;.} \ar@{=}[d]
    \\
    M \ar[r]^-{\left(\begin{smallmatrix} 0 \\ 1_M \end{smallmatrix}\right)} & {\def\arraystretch{0.9}\begin{matrix} X(i) \\ \oplus \\ M \end{matrix}} \ar[r]^-{(\begin{smallmatrix} 1_{X(i)} & 0 \end{smallmatrix})} & X(i)\;.
  }
  } 
  \end{array}  
\end{displaymath}
To see that $\iota$ and $\pi$ are in fact morphisms of representations, i.e.~that the diagram
\begin{displaymath}
  \xymatrix{
    \s_i(M)(j) \ar[d]_-{\s_i(M)(a)=0} \ar[r]^-{\iota(j)} & C(j) \ar[d]^-{C(a)} \ar[r]^-{\pi(j)} & X(j) \ar[d]^-{X(a)}
    \\
    \s_i(M)(k) \ar[r]^-{\iota(k)} & C(k) \ar[r]^-{\pi(k)} & X(k)
  }
\end{displaymath} 
is commutative for every arrow \mbox{$a \colon\! j \to k$} in $Q$, one simply checks all four cases ($1^\circ$)--($4^\circ$) in the table above.
\end{ipg}

%\enlargethispage{10ex}

\begin{prp}
  \label{prp:varphi-is-mono}
  Let $\cM$ be any abelian category and let $i$ be any vertex in a quiver $Q$. For any objects $X \in \Rep{Q}{\cM}$ and $M \in \cM$ the following conclusions hold. 
   \begin{prt}
   
   \item Assume that $\cM$ satisfies \textnormal{AB3}. If one has $\Ext{\smash{\Rep{Q}{\cM}}}{1}{X}{\s_i(M)}=0$, then the homomorphism $\Hom{\cM}{\varphi^X_i}{M}$ is surjective. 
   
   Thus, if $\cM$ has enough injectives and $\Ext{\smash{\Rep{Q}{\cM}}}{1}{X}{\s_i(I)}=0$ for each injective $I \in \cM$, then $\varphi^X_i$ is a monomorphism.
   
   \item Assume that $\cM$ satisfies \textnormal{AB3$^*$}. If one has  \smash{$\Ext{\Rep{Q}{\cM}}{1}{\s_i(M)}{X}=0$}, then the ho\-momorphism $\Hom{\cM}{M}{\psi^X_i}$ is surjective. 
   
   Thus, if $\cM$ has enough projectives and \smash{$\Ext{\Rep{Q}{\cM}}{1}{\s_i(P)}{X}=0$} for each projective $P \in \cM$, then $\psi^X_i$ is an epimorphism.

   \end{prt}   
\end{prp}

\begin{proof} 
  \proofoftag{a} We must show that for every morphism $\alpha$ there exists a morphism $\beta$ that makes the following diagram in $\cM$ commutative:
\begin{equation}
  \label{eq:ab}
  \begin{gathered}
  \xymatrix{
    \bigoplus_{a \in \arrt{i}} X(\sou{a}) \ar[d]_-{\alpha} \ar[r]^-{\varphi^X_i} & X(i) \ar@{..>}[dl]^-{\beta}
    \\
    \,M\,. & {}
  }
  \end{gathered}  
\end{equation}  
We write \smash{$\varepsilon_a \colon X(\sou{a}) \to \bigoplus_{a \in \arrt{i}} X(\sou{a})$} for the canonical injections and apply \ref{cone} to the morphisms $\xi_a= \alpha\varepsilon_a$ to obtain the short exact sequence \eqref{ses}. As \smash{$\Ext{\Rep{Q}{\cM}}{1}{X}{\s_i(M)}=0$} this sequence splits, and hence there is a morphism $\sigma \colon X \to C$ in $\Rep{Q}{\cM}$ which is a right-inverse of $\pi$. Recall that for $j \in Q_0$ with $j \neq i$ we have $\pi(j)=1_{X(j)}$ and, consequently, $\sigma(j)=1_{X(j)}$ as well. The morphism $\sigma(i)$ has two coordinate maps, say,
\begin{displaymath}
\xymatrix@C=4pc{X(i) \ar[r]^-{\sigma(i) \,=\, \left(\begin{smallmatrix} \gamma \\ \beta \end{smallmatrix}\right)} & C(i) = {\def\arraystretch{0.9} \begin{matrix} X(i) \\ \oplus \\ M \end{matrix}}. } 
\end{displaymath}
As $\sigma(i)$ is a right-inverse of $\pi(i)$, it follows that
\begin{displaymath}
  1_{X(i)} = \pi(i)\sigma(i) = \begin{pmatrix} 1_{X(i)} & \!\!\!\!0 \end{pmatrix}\!\!
\begin{pmatrix} \gamma \\ \beta \end{pmatrix} = \gamma\;.
\end{displaymath}
Since $\sigma \colon X \to C$ is a morphism of representations, we have for every arrow $a \in \arrt{i}$, say, \mbox{$a \colon\! j \to i$}, a commutative diagram:

%%%%%%%%%%%%%%%%%%%%%%%%%%%%%%

\begin{displaymath}
  \begin{array}{ccc}
  \hspace*{-10ex}\textnormal{For $j\neq i$, see \ref{cone}($3^\circ$):}  & & 
  \hspace*{-10ex}\textnormal{For $j= i$, see \ref{cone}($4^\circ$):}
  \vspace*{0.5ex} \\
  {
  \xymatrix@C=3.3pc{
    X(j) \ar[d]_-{X(a)} \ar[r]^-{\sigma(j)\,=\,1_{X(j)}} & C(j) = X(j)
    \ar@<3ex>[d]^-{C(a) \,=\, \left(\begin{smallmatrix} X(a) \\ \alpha\varepsilon_a \end{smallmatrix}\right)}
    \\
    X(i) 
    \ar[r]^-{\sigma(i) \,=\, \left(\!\!\begin{smallmatrix} 1_{X(i)} \\ \beta \end{smallmatrix}\!\!\right)} 
    & 
    C(i)= {\def\arraystretch{0.9}\begin{matrix} X(i) \\ \oplus \\ M \end{matrix}}
  }
  }
  & & 
  {
  \xymatrix@R=1.05pc@C=3.3pc{
    X(i) \ar[d]_-{X(a)} 
    \ar[r]^-{\sigma(i) \,=\, \left(\!\!\begin{smallmatrix} 1_{X(i)} \\ \beta \end{smallmatrix}\!\!\right)} 
    & C(i) = {\def\arraystretch{0.9}\begin{matrix} X(i) \\ \oplus \\ M \end{matrix}}
    \ar@<3ex>[d]^-{C(a) \,=\, \left(\begin{smallmatrix} X(a) & 0 \\ \alpha\varepsilon_a & 0 \end{smallmatrix}\right)}
    \\
    X(i) 
    \ar[r]^-{\sigma(i) \,=\, \left(\!\!\begin{smallmatrix} 1_{X(i)} \\ \beta \end{smallmatrix}\!\!\right)} 
    & 
    C(i)= {\def\arraystretch{0.9}\begin{matrix} X(i) \\ \oplus \\ M \end{matrix}}
  }
  }   
  \end{array}
\end{displaymath}  
In either case, it follows that $\beta X(a) = \alpha \varepsilon_a$. By the definition \ref{c} of $\varphi^X_i$ we have $X(a) = \varphi^X_i\varepsilon_a$, and hence $\beta \varphi^X_i\varepsilon_a = \alpha \varepsilon_a$ for all $a \in \arrt{i}$. By the universal property of the coproduct, it follows that $\beta \varphi^X_i = \alpha$, so \eqref{ab} is commutative as desired.

  \proofoftag{a} Dual to the proof of (a). Alternatively, apply part (a) directly to the opposite quiver $Q^\mathrm{op}$ and the opposite category $\cM^\mathrm{op}$.
\end{proof}

\enlargethispage{3ex}

\section{Cotorsion pairs}
\label{sec:cotorsion}

We collect some results about cotorsion pairs in abelian categories that we will need. In this section, $\cM$ is any abelian category.

For objects $M,N \in \cM$ and an integer $n \geqslant 0$ we denote by $\Ext{\cM}{n}{M}{N}$ the $n^\mathrm{th}$ Yoneda Ext group, whose elements are equivalence classes of $n$-exten\-sions of $N$ by $M$. It is well-known that if  $\cM$ has enough projectives, respectively, enough injectives, then $\Ext{\cM}{n}{M}{N}$ can be computed by using a projective resolution of $M$, respectively, an injective resolution of $N$, see e.g.~Hilton--Stammbach \cite[Chap.~IV\S9]{HiltonStammbach}.

For a class $\cC$ of objects in $\cM$ and $n \geqslant 1$, we set
\begin{align*}
  \cC^{\perp_n} &= \{\,N \in \cM \,|\, \Ext{\cM}{n}{C}{N}=0 \,\text{ for all }\, C \in \cC\,\} \quad \text{and}
  \\
  {}^{\perp_n}\cC &= \{\,M \in \cM \,|\, \Ext{\cM}{n}{M}{C}=0 \,\text{ for all }\, C \in \cC\,\}\;.  
\end{align*}
We set $\cC^\perp=\cC^{\perp_1}$ and $\cC^{\perp_\infty}=\bigcap_{\mspace{1mu}n=1}^\infty\cC^{\perp_n}$ and similarly  ${}^\perp\cC={}^{\perp_1}\cC$ and ${}^{\perp_\infty}\cC=\bigcap_{\mspace{1mu}n=1}^\infty{}^{\perp_n}\cC$.

A \emph{cotorsion pair} in $\cM$ is a pair $(\cA,\cB)$ of classes of objects in $\cM$ for which equalities $\cA^\perp = \cB$ and $\cA = {}^\perp\cB$ hold.

For a class $\cC$ of objects in $\cM$, the cotorsion pair \emph{generated} by $\cC$ is $\mathfrak{G}_{\cC} = ({}^\perp(\cC^\perp),\cC^\perp)$ and the cotorsion pair \emph{cogenerated} by $\cC$ is $\mathfrak{C}_{\cC} = ({}^\perp\cC,({}^\perp\cC)^\perp)$. We use here the terminology of G{\"o}bel and Trlifaj \cite[Def.~2.2.1]{GobelTrlifaj}. Beware that some authors---e.g.~Enochs and Jenda \cite[Def.~7.1.2]{rha} and {\v{S}}aroch and Trlifaj \cite[Introduction]{SarochTrlifaj}---use the term ``generated'' (respectively, ``cogenerated'') for what we have called ``cogenerated'' (respectively, ``generated'').

The following terminology is standard; see for example \cite[Def.~2.2.8]{GobelTrlifaj}.

\begin{ipg}
  \label{resolving}
  Let $\cC$ be a class of objects in $\cM$. If $\cM$ has enough projectives (respectively, enough injectives), then $\cC$ is called \emph{resolving} (respectively, \emph{coresolving}) if it contains all projective (respectively, all injective) objects in $\cM$ and is closed under extensions and kernels of epimorphisms (respectively, extensions and cokernels of monomorphisms).
\end{ipg}

\begin{ipg}
  \label{hereditary}
  A cotorsion pair $(\cA,\cB)$ in $\cM$ is called \emph{hereditary} if $\Ext{\cM}{n}{A}{B}=0$ for all $A \in \cA$, $B \in \cB$, and all $n \geqslant 1$. That is, $(\cA,\cB)$ is hereditary if $\cA^{\perp_\infty} \supseteq \cB$, equivalently, if $\cA \subseteq {}^{\perp_\infty}\cB$, and in the affirmative case one has $\cA^{\perp_\infty} = \cB$ and $\cA = {}^{\perp_\infty}\cB$.

  A result by Garc{\'{\i}}a Rozas' \cite[Thm.~1.2.10]{GR99} (see also \cite[Lem.~2.2.10]{GobelTrlifaj}) asserts that for~a cotorsion pair $(\cA,\cB)$ in the category  \mbox{$\cM=\mathrm{Mod}\mspace{3mu}R$} of (left) modules over a ring $R$, the fol\-lowing conditions are equivalent:
  \begin{eqc}
  \item $(\cA,\cB)$ is hereditary.
  \item $\cA$ is resolving (see \ref{resolving}).
  \item $\cB$ is coresolving (see \ref{resolving}).  
  \end{eqc}      
  An inspection of the proof of this result reveals that ($i$)$\,\Leftrightarrow\,$($ii$) holds in any abelian category $\cM$ with enough projectives and, similarly, ($i$)$\,\Leftrightarrow\,$($iii$) holds if $\cM$ has enough injectives.
\end{ipg}

\begin{ipg}
  \label{complete}
  A cotorsion pair $(\cA,\cB)$ in $\cM$ is \emph{complete} if it satisfies the following two conditions:
  \begin{eqc}
  \item The cotorsion pair $(\cA,\cB)$ \emph{has enough projectives}, that is, for every $M \in \cM$ there exists an exact sequence $0 \to B \to A \to M \to 0$ with $A \in \cA$ and $B \in \cB$.
  \item The cotorsion pair $(\cA,\cB)$ \emph{has enough injectives}, that is, for every $M \in \cM$ there exists an exact sequence $0 \to M \to B \to A \to 0$ with $A \in \cA$ and $B \in \cB$.  
  \end{eqc}  
  Salce's Lemma (which goes back to \cite{Salce}) asserts that ($i$) and ($ii$) are equivalent in the~case where \mbox{$\cM=\mathrm{Ab}$} is the category of abelian groups. The proof of this lemma, see for example \cite[Prop.~7.1.7]{rha} or \cite[Lem.~2.2.6]{GobelTrlifaj}, shows that if the abelian category $\cM$ has enough injectives, then ($i$)$\,\Rightarrow\,$($ii$); and if $\cM$ has enough projectives, then ($ii$)$\,\Rightarrow\,$($i$).
  
  Let $\cM$ be a Grothendieck category.
If $(\cA,\cB)$ is a cotorsion pair in $\cM$ generated by a set (as opposed to a proper class), then \cite[Prop.~5.8]{Stovicek2013} (or \cite[Thm.~10]{PCEJTr01} in the special case where \mbox{$\cM = \mathrm{Mod}\mspace{3mu}R$}) implies that $(\cA,\cB)$ satisfies condition ($ii$) above. As already noted, ($i$) follows from ($ii$) if $\cM$ has enough projectives, so we get:

\emph{If $\cM$ is a Grothendieck category with enough projectives, then every cotorsion pair in $\cM$ which is generated by a set is complete.}\footnote{\,Actually, one does not need to assume that the Grothendieck category $\cM$ has enough~projectives, indeed, by \cite[Thm.~5.16]{Stovicek2013} it is enough that the left half of the cotorsion pair contains a generator~of~$\cM$.}

Under certain assumptions, including G{\"o}del's Axiom of Constructibility (V$\mspace{2mu}=\mspace{2mu}$L), cotorsion pairs in \mbox{$\mathrm{Mod}\mspace{3mu}R$} that are cogenerated by a set will also be complete; see {\v{S}}aroch and Trlifaj \cite[Thms.~1.3 and 1.7]{SarochTrlifaj}.
\end{ipg}

\begin{ipg}
  \label{direct-lambda-sequence}
   Let $\lambda$ be an ordinal. A $\lambda$-direct system $\{f_{\beta\alpha} \colon M_\alpha \to M_\beta\}_{\alpha \leqslant \beta \leqslant \lambda}$ in $\cM$, that is, a well-ordered direct system in $\cM$ indexed by $\lambda$, can be (partially) illustrated as follows:
  \begin{equation}
    \label{eq:direct-lambda-sequence}
    \begin{gathered}
    \xymatrix{
      M_0 \ar@/_1.3pc/[rr]^-{f_{20}} \ar@/_1.8pc/[rrr]^(0.6){f_{30}} \ar@/_2.8pc/[rrrrr]^-{f_{\omega,0}} \ar@/_3.7pc/[rrrrrr]^(0.6){f_{\omega+1,0}} \ar[r]^-{f_{10}} & M_1 \ar[r]^-{f_{21}} & M_2 \ar[r]^-{f_{32}} & M_3 \ar[r]^-{f_{43}} &\cdots 
      \ar[r] & M_\omega \ar[r]^-{f_{\omega+1,\omega}} & M_{\omega+1} \ar[r]^-{f_{\omega+2,\omega+1}} & \cdots
    }. 
    \end{gathered}    
   \vspace*{1ex}    
  \end{equation}
  Such a system is called a \emph{direct $\lambda$-sequence} if for each limit ordinal $\mu\leqslant\lambda$, the object $M_\mu$ together with the morphisms $f_{\mu\alpha} \colon M_\alpha \to M_\mu$ for $\alpha<\mu$, is a colimit of the 
direct subsystem $\{f_{\beta\alpha} \colon M_\alpha \to M_\beta\}_{\alpha \leqslant \beta <\mu}$. In symbols: \smash{$M_\mu = \varinjlim_{\alpha<\mu}M_\alpha$}.
  
  A \emph{continuous direct $\lambda$-sequence} is a direct $\lambda$-sequence \eqref{direct-lambda-sequence} for which all the morphisms $f_{\beta\alpha} \colon M_\alpha \to M_\beta$ ($\alpha \leqslant \beta \leqslant \lambda$) are monic. 

  A \emph{$\cC$-filtration} of an object $M \in \cM$ is a continuous direct $\lambda$-sequence \eqref{direct-lambda-sequence} with $M_0=0$ and $M_\lambda=M$ and such that $\Coker{f_{\alpha+1,\alpha}} \in \cC$ for all $\alpha<\lambda$.
\end{ipg}

\begin{rmk}
\label{rmk:Stovicek}
In the paper \cite{Stovicek2013} by {\v{S}}{\v{t}}ov{\'{\i}}{\v{c}}ek, cotorsion pairs are studied in the context of exact categories. We are only dealing with abelian categories\footnote{\,Every abelian category has a canonical structure as an exact category in which all short exact sequences are considered to be conflations (hence the inflations are exactly the monomorphisms and the deflations are exactly the epimorphisms).}, but even for such categories, our definition of a $\cC$-filtration is stronger than the one found in \cite[Def.~3.7]{Stovicek2013}. Indeed, in \emph{loc.~cit.}~it is only required that the morphisms $f_{\alpha+1,\alpha} \colon M_\alpha \to M_{\alpha+1}$ are inflations (in our case,  monomorphisms) with $\Coker{f_{\alpha+1,\alpha}} \in \cC$---not that \emph{all} the morphisms $f_{\beta\alpha} \colon M_\beta \to M_\alpha$ are inflations (= monomorphisms). However, several of the results about $\cC$-filtrations found in \cite{Stovicek2013} (for example, Lem.~3.10 and Prop.~5.7) require the exact category in which the result takes place to satisfy the axiom (Ef1), which means that arbitrary transfinite compositions, in the sense of \cite[Def.~3.2]{Stovicek2013}, of inflations (= monomorphisms) exist and are themselves inflations (= monomorphisms). In such a category, \emph{all} morphisms $f_{\beta\alpha} \colon M_\beta \to M_\alpha$ in a $\cC$-filtration in the sense of {\v{S}}{\v{t}}ov{\'{\i}}{\v{c}}ek \cite[Def.~3.7]{Stovicek2013} are actually inflations (= monomorphisms). In other words, in an abelian category satisfying (Ef1), there is no difference between our definition \ref{direct-lambda-sequence} of a $\cC$-filtration and the one found in {\v{S}}{\v{t}}ov{\'{\i}}{\v{c}}ek \cite[Def.~3.7]{Stovicek2013}.
\end{rmk}

In the case where \mbox{$\cM=\mathrm{Mod}\mspace{3mu}R$} is the category of (left) modules over a ring $R$, the next result, which is called ``Eklof's Lemma'', is indeed due to Eklof \cite[Thm.~1.2]{Eklof77} or \cite[Lem.~1]{PCEJTr01}. 

  If $\cM$ is an exact category satisfying (Ef1), then \lemref{Eklof} can be found in {\v{S}}{\v{t}}ov{\'{\i}}{\v{c}}ek~\cite[Prop.~5.7]{Stovicek2013} (see also Saor{\'{\i}}n and {\v{S}}{\v{t}}ov{\'{\i}}{\v{c}}ek \cite[Prop.~2.12]{SaorinStovicek}). In our version of Eklof's Lemma (\lemref[]{Eklof} below) we are working with any cocomplete abelian category $\cM$, and such a category does not necessarily satisfy (Ef1) (as $\cM$ is cocomplete, we do have that transfinite compositions of monomorphisms exist, but the resulting composition is not necessarily monic). However, as discussed above, we are also working with a stronger meaning of the notion ``filtration'' compared to {\v{S}}{\v{t}}ov{\'{\i}}{\v{c}}ek \cite{Stovicek2013}, and this makes up for the lack of (Ef1).

\begin{lem}[Eklof]
  \label{lem:Eklof}
  Let $\cM$ be a cocomplete abelian category. Let $\cC$ be a class of objects in $\cM$ and let $M$ be an object in $\cM$. If $M$ has a ${}^\perp\cC$-filtration, then $M$ belongs to ${}^\perp\cC$.
\end{lem}

\begin{proof}
  We leave it to the reader to verify that the proof of \cite[Thm.~7.3.4]{rha} (which deals with the case \mbox{$\cM=\mathrm{Mod}\mspace{3mu}R$}) also works in the present more general setting. Here we just note that, as in the proof of \cite[Thm.~7.3.4]{rha}, we can form the preimage $g^{-1}(M_\alpha)$ of the subobject~$M_\alpha \subseteq M_\beta$ with respect to the morphism $g \colon G \to M_\beta$. Indeed, first of all, $M_\alpha$~really is a sub\-object of $M_\beta$, that is, the morphism $M_\alpha \to M_\beta$ is monic, since this is part of what it means to be a filtration in our sense \ref{direct-lambda-sequence}. Hence we can define the preimage $g^{-1}(M_\alpha)$ to be the kernel of the composite morphism \smash{$G \stackrel{g}{\twoheadrightarrow} M_\beta \twoheadrightarrow M_\beta/M_\alpha$}. 
\end{proof}

\begin{ipg}
  \label{inverse-lambda-sequence}
   Let $\lambda$ be an ordinal. A $\lambda$-inverse system $\{g_{\alpha\beta} \colon M_\beta \to M_\alpha\}_{\alpha \leqslant \beta \leqslant \lambda}$ in $\cM$, that is, a well-ordered inverse system in $\cM$ indexed by $\lambda$, can be (partially) illustrated as follows:
  \begin{equation}
    \label{eq:inverse-lambda-sequence}
    \begin{gathered}
    \xymatrix{
       \cdots \ar[r]^-{g_{\omega+1,\omega+2}} & M_{\omega+1} \ar@/_3.7pc/[rrrrrr]^(0.4){g_{0,\omega+1}} \ar[r]^-{g_{\omega,\omega+1}} & M_\omega \ar@/_2.8pc/[rrrrr]^-{g_{0,\omega}} \ar[r] & \cdots \ar[r]^-{g_{34}} & M_3 \ar@/_1.8pc/[rrr]^(0.4){g_{03}} \ar[r]^-{g_{23}} & M_2 \ar@/_1.3pc/[rr]^-{g_{02}} \ar[r]^-{g_{12}} & M_1 \ar[r]^-{g_{01}} & M_0
    }.
    \end{gathered}        
    \vspace*{1ex}
  \end{equation}   
  Such a system is called an \emph{inverse $\lambda$-sequence} if for each limit ordinal $\mu\leqslant\lambda$, the object $M_\mu$ together with the morphisms $g_{\alpha\mu} \colon M_\mu \to M_\alpha$ for $\alpha<\mu$, is a limit of the 
inverse subsystem $\{g_{\alpha\beta} \colon M_\beta \to M_\alpha\}_{\alpha \leqslant \beta <\mu}$. In symbols: \smash{$M_\mu = \varprojlim_{\alpha<\mu}M_\alpha$}.
  
  A \emph{continuous inverse $\lambda$-sequence} is an inverse $\lambda$-sequence \eqref{inverse-lambda-sequence} for which all the morphisms $g_{\alpha\beta} \colon M_\beta \to M_\alpha$ ($\alpha \leqslant \beta \leqslant \lambda$) are epic.

  A \emph{$\cC$-cofiltration} of an object $M \in \cM$ is a continuous inverse $\lambda$-sequence \eqref{inverse-lambda-sequence} with $M_0=0$ and $M_\lambda=M$ and such that $\Ker{g_{\alpha,\alpha+1}} \in \cC$ for all $\alpha<\lambda$.
\end{ipg}

In the case where \mbox{$\cM=\mathrm{Mod}\mspace{3mu}R$} is the category of (left) modules over a ring $R$, the next result is due to Trlifaj \cite[Lem.~2.3]{Trlifaj2003}. Having established the above version (\lemref[]{Eklof}) of Eklof's Lemma, the following more general version of Trlifaj's result can be inferred directly from \lemref{Eklof} by duality.

\begin{lem}[Trlifaj]
  \label{lem:Trlifaj}
  Let $\cM$ be a complete abelian category. Let $\cC$ be a class of objects in $\cM$ and let $M$ be an object in $\cM$. If $M$ has a $\cC^\perp$-cofiltration, then $M$ belongs to $\cC^\perp$.  
\end{lem}

\begin{proof}
  Consider $M$ as an object and $\cC$ as a class of objects in the opposite category $\cM^\mathrm{op}$ (which is cocomplete as $\cM$ is complete). The given $\cC^\perp$-cofiltration of $M$ in $\cM$ yields a ${}^\perp\cC$-filtration of $M$ in $\cM^\mathrm{op}$, so by \lemref{Eklof} we get that $M$ belongs to ${}^\perp\cC$ in $\cM^\mathrm{op}$, which is nothing but $\cC^\perp$ in $\cM$.
\end{proof}

\section{Cotorsion pairs in the category of quiver representations}
\label{sec:main}

In this section, $Q$ is any quiver and $\cM$ is any abelian category. 

\begin{dfn}
  \label{dfn:classes}
  For a class $\cC$ of objects in $\cM$ we set 
  \begin{align*}
  \f_*(\cC) &= \{\f_i(C) \,| \ C \in \cC \text{ and } i \in Q_0 \}\,, \\
  \g_*(\cC) &= \{\g_i(C) \,| \ C \in \cC \text{ and } i \in Q_0 \}\,, \text{ and} \\
  \s_*(\cC) &= \{\s_i(C) \,| \ C \in \cC \text{ and } i \in Q_0 \}\,.
\end{align*}
Here $\f_i$ and $\g_i$ are the left and right adjoints of the evaluation functor $\e_i$ (provided that they exist, see~\thmref{adjoints-of-e}) and $\s_i$ is the stalk functor (see \ref{e-and-s}). We also set  
  \begin{align*}
    \Rep{Q}{\cC} &= \big\{\mspace{1mu}X \in \Rep{Q}{\cM} \,\big|\, \text{$X(i) \in \cC$ for all $i \in Q_0$} \big\}\,,
    \\   
    \upPhi(\cC) &= \left\{\!
    X \in \Rep{Q}{\cM} \left|\!\!\!
    \begin{array}{ll}
      \text{$\varphi^X_i$ is a monomorphism and} \\
      \text{$\Coker{\varphi^X_i} \in \cC$ for all $i \in Q_0$} 
    \end{array}
    \right.
    \!\!\!\!\!\right\}, \quad \text{and}
    \\    
    \upPsi(\cC) &= 
    \left\{\!
    X \in \Rep{Q}{\cM} \left|\!\!\!  
    \begin{array}{ll}
      \text{$\psi^X_i$ is an epimorphism and} \\
      \text{$\Ker{\psi^X_i} \in \cC$ for all $i \in Q_0$} 
    \end{array}
    \right.
    \!\!\!\!\!\right\}.    
  \end{align*}  
\end{dfn}

Note that \emph{a priori} the classes $\upPhi(\cA)$ and $\upPsi(\cB)$ from Theorem A (where $Q$ is left rooted) and Theorem B (where $Q$ is right rooted) in the Introduction (\secref{Introduction}) look different from what we have defined above. Indeed, representations in $\upPhi(\cA)$ as defined in the Introduction must satisfy $X(i) \in \cA$ for all $i \in Q_0$. However, as explained by the next result, this seeming difference is not real. Recall that left and right rooted quivers are defined in \ref{left-rooted} and \ref{right-rooted}.

\begin{prp}
  \label{prp:values}
  Let $\cM$ be an abelian category that satisfies \textnormal{AB3} and \textnormal{AB3$^*$}, and let $\cC$ be a class of objects in $\cM$.
  
  \begin{prt}
  
  \item If the quiver $Q$ is left rooted and if $\cC$ is closed under extensions and coproducts~in~$\cM$, then every $X \in \upPhi(\cC)$ has values in $\cC$, that is, $X(i) \in \cC$ for all $i \in Q_0$.

  \item If the quiver $Q$ is right rooted and if $\cC$ is closed under extensions and products~in~$\cM$, then every $X \in \upPsi(\cC)$ has values in $\cC$, that is, $X(i) \in \cC$ for all $i \in Q_0$.
  
  \end{prt}
\end{prp}

\begin{proof}
  \proofoftag{a} Let $\{V_\alpha\}$ be the transfinite sequence of subsets of $Q_0$ from \ref{left-rooted}. Since $Q$ is left rooted we have $V_\lambda=Q_0$ for some ordinal $\lambda$. Thus, it suffices to prove the assertion
  \begin{displaymath}
  \tag{$P_\alpha$}
  \textnormal{For all $i \in V_\alpha$ and all $X \in \upPhi(\cC)$ one has $X(i) \in \cC$}
  \end{displaymath}
  for every ordinal $\alpha$. We do this by transfinite induction. The assertion ($P_0$) is true as $V_0=\emptyset$. If $\alpha$ is a limit ordinal and if ($P_\beta$) holds for all $\beta<\alpha$, then ($P_\alpha$) holds as well since, in this case, one has $V_\alpha = \bigcup_{\beta<\alpha}V_\beta$. Finally assume that $\alpha+1$ is a successor ordinal and that ($P_\alpha$) holds. We must prove that ($P_{\alpha+1}$) also holds. Let $i \in V_{\alpha+1}$ and let $X \in \upPhi(\cC)$ be given. As $\varphi^X_i$ is a monomorphism, there is a short exact sequence,
  \begin{displaymath}
    \xymatrix{
      0 \ar[r] & \bigoplus_{a \in \arrt{i}} X(\sou{a}) \ar[r]^-{\varphi^X_i} & 
      X(i) \ar[r] & \Coker{\varphi^X_i} \ar[r] & 0
    }.
  \end{displaymath}    
  Since $i \in V_{\alpha+1}$ it follows from \corref{no-arrow} that $\sou{a} \in V_\alpha$ for every $a \in \arrt{i}$, so by the induction hypothesis ($P_\alpha$) and the assumption that $\cC$ is closed under coproducts, we get that \smash{$\bigoplus_{a \in \arrt{i}} X(\sou{a})$} belongs to $\cC$. We also have $\Coker{\varphi^X_i} \in \cC$, and since $\cC$ is closed under extensions, we conclude that $X(i) \in \cC$, as desired.
  
  \proofoftag{b} Dual to (a).  
\end{proof}

With the notation from \dfnref{classes}, the results in \secref{Ext} enable us to compute the following perpendicular classes in the category $\Rep{Q}{\cM}$.

\begin{prp}
  \label{prp:Ext-in-Rep}
  Let $\cC$ be a class of objects in an abelian category $\cM$.
  \begin{prt}
  \item If $\cM$ satisfies \textnormal{AB4}, then one has $\f_*(\cC)^{\perp_n} = \Rep{Q}{\cC^{\perp_n}}$.
    
  \item If $\cM$ satisfies \textnormal{AB4$^*$}, then one has ${}^{\perp_n}\g_*(\cC) = \Rep{Q}{{}^{\perp_n}\cC}$.    
   
  \item If $\cM$ satisfies \textnormal{AB3} and has enough injectives and $\cC \supseteq \Inj{\cM}$, then \mbox{${}^\perp\s_*(\cC)=\upPhi({}^\perp\cC)$}.
  
  \item If $\cM$ satisfies \textnormal{AB3$^*$} and has enough projectives and $\cC \supseteq \Proj{\cM}$, then \mbox{$\s_*(\cC)^\perp = \upPsi(\cC^\perp)$}.

  \end{prt}
\end{prp}

\begin{proof}
  Parts (a) and (b) follow immediately from \prpref{Ext-e}. In part (c), the inclusion ``$\supseteq$'' follows from \prpref{Ext-s}(a), and the opposite inclusion ``$\subseteq$'' follows from \prpref[Propositions~]{varphi-is-mono}(a) and \prpref[]{Ext-s}(a). Similarly, (d) follows from \prpref[Propositions~]{Ext-s}(b) and \prpref[]{varphi-is-mono}(b).
\end{proof}

\begin{thm}
  \label{thm:4xcotorsion}
 Let $\cM$ be an abelian category that satisfies \textnormal{AB4} and \textnormal{AB4*} and which has enough projectives and injectives. Let $(\cA,\cB)$ be a cotorsion pair in $\cM$ which is generated by a class $\cA_0$ (e.g.~$\cA_0=\cA$) and cogenerated by a class~$\cB_0$ (e.g.~$\cB_0=\cB$). 
  
  \begin{prt}
  
  \item The cotorsion pair in $\Rep{Q}{\cM}$ generated by $\f_*(\cA_0)$ is 
  \begin{displaymath}  
     \mathfrak{G}_{\f_*(\cA_0)} = ({}^\perp\Rep{Q}{\cB},\Rep{Q}{\cB})\;.
  \end{displaymath}   
  If $\cB_0 \supseteq \Inj{\cM}$, then the cotorsion pair in $\Rep{Q}{\cM}$ cogenerated by $\s_*(\cB_0)$ is
  \begin{displaymath}  
    \mathfrak{C}_{\s_*(\cB_0)} = (\upPhi(\cA),\upPhi(\cA)^\perp)\;.
  \end{displaymath}
  
  \item The cotorsion pair in $\Rep{Q}{\cM}$ cogenerated by $\g_*(\cB_0)$ is 
  \begin{displaymath}  
     \mathfrak{C}_{\g_*(\cB_0)} = (\Rep{Q}{\cA},\Rep{Q}{\cA}^\perp)\;.
  \end{displaymath}      
  If $\cA_0 \supseteq \Proj{\cM}$, then the cotorsion pair in $\Rep{Q}{\cM}$ generated by $\s_*(\cA_0)$ is
    \begin{displaymath}  
      \mathfrak{G}_{\s_*(\cA_0)} = ({}^\perp\upPsi(\cB),\upPsi(\cB))\;.
  \end{displaymath}  

    \end{prt}  
\end{thm}

\begin{proof}
  Part (a) follows from \prpref{Ext-in-Rep}(a,c), and (b) from \prpref{Ext-in-Rep}(b,d).
\end{proof}

\begin{rmk}
  \label{rmk:generated}
  If $\cA_0$, respectively, $\cB_0$, is a set, then so is $\f_*(\cA_0)$, respectively, $\g_*(\cB_0)$. Thus, if the cotorsion pair $(\cA,\cB)$ is generated by a set, then so is $({}^\perp\Rep{Q}{\cB},\Rep{Q}{\cB})$, and if $(\cA,\cB)$ is cogenerated by a set, then so is $(\Rep{Q}{\cA},\Rep{Q}{\cA}^\perp)$. %Such cotorsion pairs are ``often'' complete; see \ref{complete}.
\end{rmk}

We will shortly show (\thmref{cotorsion-pairs-agree} below) that if $Q$ is left rooted, then the two cotorsion pairs in part (a) of the theorem above are the same and, similarly, if $Q$ is right rooted, then the two cotorsion pairs in part (b) are the same.

Suppose that the cotorsion pair $(\cA,\cB)$ has a certain property, for example, $(\cA,\cB)$ could be hereditary or complete. It is then natural to ask if the induced cotorsion pairs in \thmref{4xcotorsion} have the same property.

\begin{prp}
  \label{prp:hereditary}
  Adopt the setup and the notation from \thmref{4xcotorsion}. If the cotorsion pair $(\cA,\cB)$ is hereditary, then so are all four cotorsion pairs in \thmref{4xcotorsion}.
\end{prp}

\begin{proof}
  Recall from \corref{enough-projectives} that the abelian category $\Rep{Q}{\cM}$ has enough projec\-tives and enough injectives, so by \ref{hereditary} we only need to show that if $\cA$ is resolving, then so are $\Rep{Q}{\cA}$ and $\upPhi(\cA)$, and if $\cB$ is coresolving, then so are $\Rep{Q}{\cB}$ and $\upPsi(\cB)$.
  
  If $\cA$ is resolving, then clearly so is $\Rep{Q}{\cA}$. To see that $\upPhi(\cA)$ is resolving, note that $\upPhi(\cA)$ is closed under extensions and contains all projective objects in $\Rep{Q}{\cM}$ as $\upPhi(\cA)$ is the left half of a cotorsion pair. It remains to see that if $0 \to X' \to X \to X'' \to 0$ is a short exact sequence in $\Rep{Q}{\cM}$ with $X,X'' \in \upPhi(\cA)$, then one also has $X' \in \upPhi(\cA)$. To this end, consider for every $i \in Q_0$ the following commutative diagram with exact rows:
\begin{displaymath}
  \xymatrix@C=1.3pc{
    0 \ar[r]
    & 
    \bigoplus_{a \in \arrt{i}} X'(\sou{a}) \ar[d]^-{\varphi^{X'}_i} \ar[r] 
    & 
    \bigoplus_{a \in \arrt{i}} X(\sou{a}) \ar[d]^-{\varphi^X_i} \ar[r]
    & 
    \bigoplus_{a \in \arrt{i}} X''(\sou{a}) \ar[d]^-{\varphi^{X''}_i} \ar[r] 
    & 0
    \\
    0 \ar[r] & X'(i) \ar[r] & X(i) \ar[r] & X''(i) \ar[r] & 0.
  }
\end{displaymath}   
By assumption, \smash{$\varphi^X_i$} and \smash{$\varphi^{X''}_i$} are monomorphisms with cokernels in $\cA$. From the Snake Lemma and the assumption that $\cA$ is resolving, it now follows that \smash{$\varphi^{X'}_i$} is a monomorphism with cokernel in $\cA$. Since this is true for every $i \in Q_0$ we conclude that $X' \in \upPhi(\cA)$.

Similar arguments show that if $\cB$ is coresolving, then so are $\Rep{Q}{\cB}$ and $\upPsi(\cB)$.
\end{proof}

As mentioned in \ref{complete}, if the category $\cM$ is Grothendieck with enough projectives and the cotorsion pair $(\cA,\cB)$ is generated by a set, then it is also complete. If $(\cA,\cB)$ is complete for this strong reason, then the induced cotorsion pair \smash{$({}^\perp\Rep{Q}{\cB},\Rep{Q}{\cB})$}---which by \thmref{cotorsion-pairs-agree} below is equal to \smash{$(\upPhi(\cA),\upPhi(\cA)^\perp)$} when $Q$ is left rooted---will also be complete, since it too is generated by a set (see \rmkref{generated}) and $\Rep{Q}{\cM}$ is Grothendieck with enough projectives (see \ref{inherited-structures} and \corref{enough-projectives}). 

Many complete cotorsion pairs in e.g.~\mbox{$\cM = \mathrm{Mod}\mspace{3mu}R$} are known to be generated by sets. For example, this is the case for the trivial cotorsion pairs \mbox{$(\Proj{R},\mathrm{Mod}\mspace{3mu}R)$} (generated by $\{0\}$) and \mbox{$(\mathrm{Mod}\mspace{3mu}R,\Inj{R})$} (generated by $\{R/\mathfrak{a} \,|\, \mathfrak{a} \subseteq R \textnormal{ ideal}\}$ beacuse of Baer's criterion). Also the flat cotorsion pair \mbox{$(\mathrm{Flat}\mspace{2.5mu}R,(\mathrm{Flat}\mspace{2.5mu}R)^\perp)$} is generated by a set, in fact, \emph{the flat cover conjecture} was settled affirmatively by proving the existence of such a generating set; see~\cite[Prop.~2]{BEE-01}.

This gives a partial answer to the following:

\begin{qst}
  Is it true that if the cotorsion pair $(\cA,\cB)$ is complete, then so are the four cotorsion pairs  in \thmref{4xcotorsion}?
\end{qst}

The next example gives a positive answer to this question in some other special cases.

\begin{exa}
  Let $Q$ be a finite quiver and let \mbox{$\cM = \mathrm{Mod}\mspace{3mu}R$}. In this case, the path ring $RQ$ is unital and the category $\Rep{Q}{\cM}$ is equivalent to \mbox{$\mathrm{Mod}\mspace{3mu}RQ$}. For a cotorsion pair $(\cA,\cB)$ in \mbox{$\cM = \mathrm{Mod}\mspace{3mu}R$} we write $(\tilde{\cA},\tilde{\cB})$ for the induced cotorsion pair $({}^\perp\Rep{Q}{\cB},\Rep{Q}{\cB})=(\upPhi(\cA),\upPhi(\cA)^\perp)$
in \mbox{$\mathrm{Mod}\mspace{3mu}RQ$}; see \thmref{cotorsion-pairs-agree}(a) below.
  
 We write $\GProj{R}$ for the class of Gorenstein projective (left) $R$-modules; see~\cite{EEnOJn95b}. Under mild assumptions on $R$ it is known that every $R$-module has a \emph{special Gorenstein projective precover} (in the sense of Xu~\cite[Prop.~2.1.3]{xu}), see e.g.~(proofs of) J{\o}rgensen~\cite[Cor.~2.13]{PJr07} and Murfet and Salarian \cite[Thm.~A.1]{DMfSSl11}, and hence $(\cA,\cB)=(\GProj{R},(\GProj{R})^\perp)$ is a complete cotorsion pair in \mbox{$\mathrm{Mod}\mspace{3mu}R$}. It not known if this cotorsion pair is generated~by~a~set! Nevertheless, in this case the induced cotorsion pair $(\tilde{\cA},\tilde{\cB})$ in \mbox{$\mathrm{Mod}\mspace{3mu}RQ$} will be complete as well, since it is nothing but $(\GProj{RQ},(\GProj{RQ})^\perp)$. This follows from \cite[Thm.~3.5.1(b)]{EshraghiHafeziSalarian}, as mentioned in the Introduction (\secref{Introduction}).
 
Similarly, under weak hypotheses, see Krause~\cite[Thm.~7.12]{HKr05}, the Gorenstein injective cotorsion pair $({}^\perp(\GInj{R}),\GInj{R})$ is complete, even though it is not known to be generated by a set. The induced cotorsion pair $(\bar{\cA},\bar{\cB})=(\Rep{Q}{\cA},\Rep{Q}{\cA}^\perp)=({}^\perp\upPsi(\cB),\upPsi(\cB))$ in \mbox{$\mathrm{Mod}\mspace{3mu}RQ$}, see \thmref{cotorsion-pairs-agree}(b) below, is also complete as it is nothing but the Goren\-stein injective cotorsion pair $({}^\perp(\GInj{RQ}),\GInj{RQ})$ in \mbox{$\mathrm{Mod}\mspace{3mu}RQ$}. See the Introduction.
\end{exa}

Recall from \ref{left-rooted} and \ref{right-rooted} the definitions of left rooted and right rooted quivers.

\begin{thm}
  \label{thm:cotorsion-pairs-agree}
  Adopt the setup and the notation from \thmref{4xcotorsion}.
  \begin{prt}

  \item If $Q$ is left rooted, then one has  
  $({}^\perp\Rep{Q}{\cB},\Rep{Q}{\cB})=(\upPhi(\cA),\upPhi(\cA)^\perp)$.
  
  \item If $Q$ is right rooted, then one has  
  $(\Rep{Q}{\cA},\Rep{Q}{\cA}^\perp)=({}^\perp\upPsi(\cB),\upPsi(\cB))$.
  
  \end{prt}
\end{thm}

\begin{proof}
 \proofoftag{a} From \thmref{4xcotorsion} we have
 \begin{displaymath}
   \Rep{Q}{\cB} = \f_*(\cA)^\perp
   \quad \text{and} \quad
   \upPhi(\cA)={}^\perp\s_*(\cB)\;,   
 \end{displaymath}   
 and it must be shown that $\Rep{Q}{\cB}=\upPhi(\cA)^\perp$. For all objects $A \in \cA$ and $B \in \cB$ and all vertices $i,j \in Q_0$ we have
 \begin{displaymath}
   \Ext{\Rep{Q}{\cM}}{1}{\f_i(A)}{\s_j(B)} \cong 
  \Ext{\cM}{1}{A}{\e_i\s_j(B)} \cong 0\;,
 \end{displaymath}   
 where the first isomorphism follows from \prpref{Ext-e}(a) and the second isomorphism follows as $(\cA,\cB)$ is a cotorsion pair and since $\e_i\s_j(B)$ is in $\cB$ (more precisely, $\e_i\s_j(B)=0$ if $i \neq j$ and $\e_i\s_j(B)=B$ if $i=j$). This shows the inclusion $\f_*(\cA) \subseteq {}^\perp\s_*(\cB)$, and consequently 
 \begin{displaymath}
   \Rep{Q}{\cB} = \f_*(\cA)^\perp \supseteq ({}^\perp\s_*(\cB))^\perp =
   \upPhi(\cA)^\perp\;.
 \end{displaymath}   
 To show the opposite inclusion, it suffices by \lemref{Trlifaj} to argue that every $Y \in \Rep{Q}{\cB}$ has a $\upPhi(\cA)^\perp$-cofiltration. To this end, let $\{V_\alpha\}$ be the transfinite sequence of subsets of $Q_0$ from \ref{left-rooted}. As $Q$ is left rooted we have $V_\lambda=Q_0$ for some ordinal $\lambda$. For any $Y \in \Rep{Q}{\cM}$ we define, for every ordinal $\alpha \leqslant \lambda$, a representation $Y_\alpha \in \Rep{Q}{\cM}$ as follows:
\begin{displaymath}
  \begin{gathered}
  Y_\alpha(i) =
  \left\{\!\!\!
    \begin{array}{cll}
      Y(i) & \text{if} & i \in V_\alpha \\
       0 & \text{if} & i \notin V_\alpha     
    \end{array}
  \right.
  \end{gathered}
  \qquad (i \in Q_0).
\end{displaymath}
For an arrow \mbox{$a \colon\! i \to j$} in $Q$ the morphism
\begin{displaymath}
  \begin{gathered}
  \xymatrix{Y_\alpha(i) \ar[r]^-{Y_\alpha(a)} & Y_\alpha(j)} \quad \text{is} \quad
  \left\{\!\!\!
    \begin{array}{cll}
      Y(a) & \text{if} & i \in V_\alpha \text{ and } j \in V_\alpha \\    
      0 & \text{if} & i \notin V_\alpha \text{ or } j \notin V_\alpha\,.
    \end{array}
  \right.
  \end{gathered}
\end{displaymath}
Note that $Y_0=0$ since $V_0=\emptyset$ and that $Y_\lambda=Y$ since $V_\lambda=Q_0$. For ordinals $\alpha\leqslant\beta\leqslant\lambda$ we define a morphism $g_{\alpha\beta} \colon Y_\beta \to Y_\alpha$ as follows:
\begin{prt}
\item[--] If $i \in V_\alpha$ ($\subseteq V_{\beta}$ by \lemref{Ascending}), then $Y_\beta(i)=Y(i)=Y_\alpha(i)$ and we set $g_{\alpha\beta}(i)=1_{Y(i)}$. 

\item[--] If $i \notin V_\alpha$ then $Y_\alpha(i)=0$ and we set $g_{\alpha\beta}(i)=0$.
\end{prt}
To see that $g_{\alpha\beta}$ really is a morphism of quiver representations, it must be argued that for every arrow \mbox{$a \colon\! i \to j$} in $Q$, the following diagram is commutative:
\begin{equation}
  \label{eq:Yalpha}
  \begin{gathered}
  \xymatrix@C=3pc{
    Y_{\beta}(i) \ar[r]^-{g_{\alpha\beta}(i)} \ar[d]_-{Y_{\beta}(a)} & Y_{\alpha}(i) \ar[d]^-{Y_{\alpha}(a)} 
    \\
    Y_{\beta}(j) \ar[r]^-{g_{\alpha\beta}(j)} & Y_{\alpha}(j).    
  }
  \end{gathered}  
\end{equation}
If $j \notin V_{\alpha}$, then $Y_{\alpha}(j)=0$ and \eqref{Yalpha} is obviously commutative. Assume that $j \in V_{\alpha}$ ($\subseteq V_\beta$). If we do have an arrow \mbox{$a \colon\! i \to j$} in $Q$, it follows from \corref{no-arrow} that we must have~$i \in V_\alpha$. In this situation, the diagram \eqref{Yalpha} looks as follows, and it is clearly commutative:
\begin{displaymath}
  \xymatrix@C=3pc{
    Y(i) \ar[r]^-{1_{Y(i)}} \ar[d]_-{Y(a)} & Y(i) \ar[d]^-{Y(a)}
    \\
    Y(j) \ar[r]^-{1_{Y(j)}} & Y(j).    
  }
\end{displaymath}
It is not hard to see that the hereby constructed system 
$\{g_{\alpha\beta} \colon Y_\beta \to Y_\alpha\}_{\alpha \leqslant \beta \leqslant \lambda}$ is a conti\-nuous inverse $\lambda$-sequence in $\Rep{Q}{\cM}$, see \ref{inverse-lambda-sequence}, and as already noted we have $Y_0=0$ and $Y_\lambda=Y$. We will show that if $Y \in \Rep{Q}{\cB}$, then this system is a $\upPhi(\cA)^\perp$-cofiltration (of $Y$), i.e.~the representation $K_\alpha:=\Ker{g_{\alpha,\alpha+1}}$ belongs to $\upPhi(\cA)^\perp$ for all $\alpha<\lambda$. Note that
\begin{displaymath}
  \begin{gathered}
  K_\alpha(i) = \Ker{(g_{\alpha,\alpha+1}(i))} =
  \left\{
    \begin{array}{cl}
      Y(i) & \text{if} \quad i \in V_{\alpha+1}\smallsetminus V_\alpha \\
       0 & \text{otherwise}    
    \end{array}
  \right.
  \end{gathered}
  \qquad (i \in Q_0).
\end{displaymath}
We claim that for every arrow \mbox{$a \colon\! i \to j$} in $Q$, the morphism $K_\alpha(a) \colon K_\alpha(i) \to K_\alpha(j)$ is zero. Indeed, if $j \notin V_{\alpha+1}\smallsetminus V_\alpha$, then $K_\alpha(j)=0$ and hence $K_\alpha(a)$ is zero. If $j \in V_{\alpha+1}\smallsetminus V_\alpha \subseteq V_{\alpha+1}$ then, if we do have an arrow \mbox{$a \colon\! i \to j$} in $Q$, it follows from \corref{no-arrow} that $i \in V_\alpha$ and hence $i \notin V_{\alpha+1}\smallsetminus V_\alpha$. Thus one has $K_\alpha(i)=0$, and therefore $K_\alpha(a)$ is also zero in this case. It follows that
\begin{displaymath}
  K_\alpha \,=\, \textstyle\prod_{i \in V_{\alpha+1}\smallsetminus V_\alpha}\s_i(Y(i))\;.
\end{displaymath}
Now, if $Y \in \Rep{Q}{\cB}$, then each $Y(i)$ belongs to $\cB$, and consequently one has
\begin{displaymath}
  \s_i(Y(i)) \in \s_*(\cB) \subseteq ({}^\perp\s_*(\cB))^\perp = \upPhi(\cA)^\perp
\end{displaymath}
for all $i \in Q_0$. Since $\upPhi(\cA)^\perp$ is closed under products in $\cM$, it follows that $K_\alpha \in \upPhi(\cA)^\perp$.

  \proofoftag{b} Dual to (a). %\enlargethispage{4ex}
\end{proof}

At this point, the proofs of Theorems A and B from the Introduction are simply a matter of collecting the appropriate references.

\begin{proof}[Proof of Theorem A]
  Since $Q$ is left rooted, \thmref{4xcotorsion}(a) yields that $(\upPhi(\cA),\Rep{Q}{\cB})$ is a cotorsion pair in $\Rep{Q}{\cM}$, where $\upPhi(\cA)$ is given in \dfnref{classes}. It follows from \prpref{values}(a) that this class $\upPhi(\cA)$ equals the class from the Introduction, which is denoted by the same symbol. The assertions about $(\upPhi(\cA),\Rep{Q}{\cB})$ being hereditary or being generated by a set follow from \prpref{hereditary} and \rmkref{generated}.
\end{proof}

\begin{proof}[Proof of Theorem B]
  Follows from \thmref{4xcotorsion}(b), \prpref{values}(b), \prpref{hereditary}, and \rmkref{generated}; cf.~the proof of Theorem A above.
\end{proof}

\section*{Acknowledgement}
  We thank the referee for a careful reading of the paper and for useful comments.

%\enlargethispage{5ex}

\bibliographystyle{amsplain}
\bibliography{+references}

\end{document}